\documentclass[11pt]{amsart}

\usepackage{amsmath,amscd,amssymb}

\usepackage{hyperref}
\hypersetup{colorlinks,linkcolor={red},citecolor={blue},urlcolor={blue}}

\evensidemargin\oddsidemargin

\newtheorem{theorem}{Theorem}[section]
\newtheorem{lemma}[theorem]{Lemma}
\newtheorem{proposition}[theorem]{Proposition}
\newtheorem{corollary}[theorem]{Corollary}

\theoremstyle{definition}
\newtheorem{definition}[theorem]{Definition}

\newtheorem{remark}[theorem]{Remark}

\numberwithin{equation}{section}

\newcommand{\CC}{\mathbb C}
\newcommand{\HH}{\mathbb H}

\newcommand{\NN}{{\mathbb N}_0}
\newcommand{\PP}{\mathbb P}
\newcommand{\QQ}{\mathbb Q}
\newcommand{\RR}{\mathbb R}
\newcommand{\ZZ}{\mathbb Z}

\newcommand{\SL}{\mathop{\mathrm {SL}}\nolimits}

\newcommand{\Orth}{\mathop{\null\mathrm {O}}\nolimits}

\newcommand{\latt}[1]{{\langle{#1}\rangle}}

\newcommand{\orb}{\mathop{\mathrm {orb}}\nolimits}

\newcommand{\Borch}{\operatorname{Borch}}

\newcommand{\m}{\operatorname{mod}}
\newcommand{\w}{\operatorname{w}}
\newcommand{\cusp}{\operatorname{cusp}}

\usepackage{mathtools}
\usepackage{tikz}
\usetikzlibrary{chains}

\tikzset{node distance=2em, ch/.style={circle,draw,on chain,inner sep=2pt},chj/.style={ch,join},line width=1pt,baseline=-1ex}

\let\dlabel=\alabel

\newcommand{\dnode}[2][chj]{%
\node[#1,label={below:\dlabel{#2}}] {};
}

\newcommand{\dnodebr}[1]{%
\node[chj,label={below right:\dlabel{#1}}] {};
}

\newcommand{\abs}[1]{\lvert#1\rvert}

\begin{document}

\title[Weyl invariant $E_8$ Jacobi forms]{Weyl invariant \texorpdfstring{$\boldsymbol{E_8}$}{E8} Jacobi forms}

\author{Haowu Wang}

\address{Max-Planck-Institut f\"{u}r Mathematik, Vivatsgasse 7, 53111 Bonn, Germany}

\email{haowu.wangmath@gmail.com}


\date{\today}


\begin{abstract}
We investigate the Jacobi forms for the root system $E_8$ invariant under the Weyl group. This type of Jacobi forms has significance in Frobenius manifolds, Gromov--Witten theory and string theory. In 1992, Wirthm\"{u}ller proved that the space of Jacobi forms for any irreducible root system not of type $E_8$ is a polynomial algebra.  But very little has been known about the case of $E_8$. In this paper we show that the bigraded ring of Weyl invariant $E_8$ Jacobi forms is not a polynomial algebra and prove that every such Jacobi form can be expressed uniquely as a polynomial in nine algebraically independent Jacobi forms introduced by Sakai with coefficients which are meromorphic $\SL_2(\ZZ)$ modular forms. The latter result implies that the space of Weyl invariant $E_8$ Jacobi forms of fixed index is a free module over the ring of $\SL_2(\ZZ)$ modular forms and that the number of generators can be calculated by a generating series. We determine and construct all generators of small index. These results give a proper extension of the Chevalley type theorem to the case of $E_8$.
\end{abstract}

\maketitle

\tableofcontents

\section{Introduction}
For the lattice constructed from the classical root system $R$, one can define the Jacobi forms which are invariant with respect to the Weyl group $W(R)$. Such Jacobi forms are called $W(R)$-invariant Jacobi forms. In this setting, the classical Jacobi forms due to Eichler and Zagier \cite{EZ} are actually the $W(A_1)$-invariant Jacobi forms. All $W(R)$-invariant weak Jacobi forms make up a bigraded ring graded by the weight and the index.  The problem on the algebraic structure of such bigraded ring was inspired by the work of E. Looijenga \cite{Loo76, Loo80} and K. Saito \cite{Sai90} on the invariants of generalized root systems. This problem is closely related to the theory of Frobenius manifolds (see \cite{Ber00a, Ber00b, Dub96, Dub98, Sat98}). The first solution to this problem was given by Wirthm\"{u}ller in 1992 \cite{Wi}, which stated that the bigraded ring of $W(R)$-invariant weak Jacobi forms is a polynomial algebra over $\CC$ except for the root system $E_8$. This result can be regarded as the Chevalley type theorem for affine root systems (see \cite{BS78, BS06}). But Wirthm\"{u}ller's solution was totally non-constructive because it did not give the construction of generators like in the case of $A_1$ considered in the book \cite{EZ} of Eichler-Zagier. On account of this defect, later Bertola \cite{Ber99} and Satake \cite{Sat98} reconsidered this problem and found explicit constructions of the generators for the root systems $A_n$, $B_n$, $G_2$, $D_4$ and $E_6$.

The case $R=E_8$ was not covered by Wirthm\"{u}ller's theorem and remains completely open for 27 years. In recent years, $W(E_8)$-invariant Jacobi forms appear in various contexts in mathematics and physics and have applications in Gromov-Witten theory and string theory. The original Seiberg-Witten curve for the $E$-string theory is expressed in terms of $W(E_8)$-invariant Jacobi forms \cite{Sa1}. The Gromov-Witten partition function $Z_{g;n}(\tau\lvert\mu)$ of genus $g$ and winding number $n$ of the local $B_9$ model is a $W(E_8)$-invariant quasi-Jacobi form of index $n$ and weight $2g-6+2n$ up to a factor \cite{M,MNVW}. The $\PP_1$-relative Gromov-Witten potentials of the rational elliptic surface are $W(E_8)$-invariant quasi-Jacobi forms numerically and the Gromov-Witten potentials of the Schoen Calabi-Yau threefold (relative to $\PP_1$) are $E_8 \times E_8$ quasi-bi-Jacobi forms \cite{OP}. But unfortunately, very little has been known about the ring of $W(E_8)$-invariant Jacobi forms. The purpose of this paper is to fill this gap in the literature.

Let $\tau \in \HH$ and $\mathfrak{z}\in E_8\otimes \CC$, a holomorphic function $\varphi(\tau,\mathfrak{z})$ is called a $W(E_8)$-invariant weak Jacobi form if it is invariant under the Jacobi group which is the semidirect product of $\SL_2(\ZZ)$ with the integral Heisenberg group of $E_8$, if it has a Fourier expansion of the form 
$$\varphi ( \tau, \mathfrak{z} )=\sum_{n\geq 0} \sum_{\ell\in E_8}f(n,\ell)e^{2\pi i (n\tau + (\ell,\mathfrak{z}))},$$
and if it satisfies $\varphi(\tau, \sigma(\mathfrak{z}))=\varphi(\tau, \mathfrak{z})$, for all $\sigma\in W(E_8)$.

Sakai \cite{Sa1} constructed nine independent $W(E_8)$-invariant holomorphic Jacobi forms, noted by $A_1$, $A_2$, $A_3$, $A_4$, $A_5$, $B_2$, $B_3$, $B_4$, $B_6$. The Jacobi forms $A_m$, $B_m$ are of weight $4$, $6$ and index $m$ respectively. Sakai \cite{Sa2} also conjectured that the number of generators of $W(E_8)$-invariant Jacobi forms of index $m$ coincides with the number of fundamental representations at level $m$. In this paper, we give an explicit mathematical description of his conjecture and prove it to be true. We first study the values of $W(E_8)$-invariant Jacobi forms at $q=e^{2\pi i \tau}=0$, which are called $q^0$-terms of Jacobi forms. We prove that the $q^0$-term of any $W(E_8)$-invariant Jacobi form can be written as a particular polynomial in terms of the Weyl orbits of eight fundamental weights of $E_8$ (see Lemma \ref{lem1}). We then deduce the following structure theorem which gives an extension of Wirthm\"{u}ller's theorem to the case of $E_8$.

\begin{theorem}[see Theorem \ref{MTH}] 
Let $t$ be a positive integer. Then the space $J_{*,E_8,t}^{\w ,W(E_8)}$ of $W(E_8)$-invariant weak Jacobi forms of index $t$ is a free module of rank $r(t)$ over the ring $M_*$ of $\SL_2(\ZZ)$ modular forms, where $r(t)$ is given by
\begin{equation*}
\frac{1}{(1-x)(1-x^2)^2(1-x^3)^2(1-x^4)^2(1-x^5)(1-x^6)}=\sum_{t\geq 0}r(t)x^t.
\end{equation*}
Equivalently, we have that
\begin{equation*}
J_{*,E_8,*}^{\w ,W(E_8)}=\bigoplus_{t\geq 0}J_{*,E_8,t}^{\w ,W(E_8)} \subsetneq \CC\left( E_4, E_6 \right) \left[A_1,A_2,A_3,A_4,A_5,B_2,B_3,B_4,B_6 \right],
\end{equation*}
and that $A_1$, $A_2$, $A_3$, $A_4$, $A_5$, $B_2$, $B_3$, $B_4$, $B_6$ are algebraically independent over $M_*$.
Here $\CC\left( E_4, E_6 \right)$ denotes the fractional field of $\CC[E_4,E_6]$.
\end{theorem}

It is well-known that $J_{*,E_8,1}^{\w ,W(E_8)}$ is generated over $M_*$ by the theta function of $E_8$. In this paper we further elucidate the structure of $J_{*,E_8,t}^{\w ,W(E_8)}$ for $t=2$, 3, 4. We next explain the main ideas.  The approach below can also be used to study the space of Jacobi forms associated to any positive-definite lattice.

Our main theorem says that the space $J_{*,E_8,t}^{\w ,W(E_8)}$ is a free module over $M_*$ and the number of generators is $r(t)$. It means that we only need to find all generators in order to characterize the structure of $J_{*,E_8,t}^{\w ,W(E_8)}$. We do this by analyzing the spaces of weak Jacobi forms of index $t$ and any non-positive weight.

We first construct many basic Jacobi forms of non-positive weight in terms of Jacobi theta functions. The main tool is the weight raising differential operators (see Lemma \ref{diffoperator}).

We then determine the dimension of the space of Jacobi forms of fixed index and any given negative weight. To do this, we study the orbits of $E_8$ vectors of fixed norm under the action of the Weyl group and represent $q^0$-terms of Jacobi forms as linear combinations of these orbits.  By means of the following two crucial facts, 
\begin{enumerate}
\item From a given Jacobi form of weight $k$, we can construct a Jacobi form of weight $k+2j$ for every positive integer $j$ by the differential operators. 
\item If one takes $\mathfrak{z}=0$, then the $q^0$-term of any Jacobi form of negative weight will be zero. In addition, the coefficients of $q^0$-term of a Jacobi form of weight zero satisfy a linear relation (see Lemma \ref{lf}).
\end{enumerate}
from a Jacobi form of negative weight with given $q^0$-term, we can build a certain system of linear equations defined by the coefficients of the orbits of $E_8$ vectors in $q^0$-terms of Jacobi forms. By solving the linear system, we can know if the above Jacobi form of negative weight with given $q^0$-term exists.  

The following theorem describes the structure of $J_{*,E_8,t}^{\w ,W(E_8)}$ for $t=2$, $3$, $4$.

\begin{theorem}[see Theorems \ref{Th index 2}, \ref{Th index 3}, \ref{Th index 4}]
\begin{align*}
J^{\w ,W(E_8)}_{*, E_8,2}&=M_* \langle \varphi_{-4,2},\varphi_{-2,2},\varphi_{0,2} \rangle,\\
J^{\w ,W(E_8)}_{*, E_8,3}&=M_* \langle \varphi_{-8,3}, \varphi_{-6,3}, \varphi_{-4,3},\varphi_{-2,3},\varphi_{0,3} \rangle,\\
J^{\w ,W(E_8)}_{*, E_8,4}&=M_* \langle \varphi_{-2k,4},\, 0\leq k \leq 8; \;\psi_{-8,4} \rangle,
\end{align*}
where $\varphi_{k,t}$ is a $W(E_8)$-invariant weak Jacobi form of weight $k$ and index $t$.
\end{theorem}

Combining the above two theorems together, we prove that the bigraded ring $J^{\w ,W(E_8)}_{*,E_8,*}$ is in fact not a polynomial algebra over $M_*$ and its structure is rather complicated (see Theorem \ref{polynomial}).  It means that the Chevalley type theorem does not hold for $E_8$.

Furthermore, we apply our results on weak Jacobi forms to study the structures of the modules of $W(E_8)$-invariant holomorphic and cusp Jacobi forms of index $2$, $3$, $4$. Our main results are as follows.

\begin{theorem}[see Theorems \ref{Th index 2}, \ref{Th index 3h}, \ref{Th index 4h}]
\begin{align*}
J^{W(E_8)}_{*, E_8,2}&=M_* \langle A_2, B_2, A_1^2 \rangle,\\
J^{W(E_8)}_{*, E_8,3}&=M_* \langle A_3, B_3, A_1A_2, A_1B_2, A_1^3 \rangle,
\end{align*}
and $J^{W(E_8)}_{*, E_8,4}$ is generated over $M_*$ by two Jacobi forms of weight $4$, two Jacobi forms of weight $6$, three Jacobi forms of weight $8$, two Jacobi forms of weight $10$ and one Jacobi form of weight $12$.
\end{theorem}

\begin{theorem}[see Theorems \ref{Th index 2}, \ref{thcusp3}, \ref{Th index 4h}]
Let $t=2$, $3$ or $4$. The numbers of generators of indicated weight of $J^{\cusp ,W(E_8)}_{*, E_8,t}$ are shown in Table \ref{tablecusp}.
\end{theorem}

\begin{table}[ht]
\caption{Number of generators of $J^{\cusp ,W(E_8)}_{*, E_8,t}$}\label{tablecusp}
\renewcommand\arraystretch{1.5}
\noindent\[
\begin{array}{|c|c|c|c|c|c|}
\hline 
\text{weight} & 8 & 10 & 12 & 14 & 16 \\ 
\hline 
t=2 & 0 & 0 & 1 & 1 & 1 \\ 
\hline 
t=3 & 0 & 1 & 2 & 1 & 1 \\ 
\hline 
t=4 & 1 & 2 & 3 & 2 & 2 \\ 
\hline 
\end{array} 
\]
\end{table}

The plan of this paper is as follows. In \S \ref{Sec:2} we recall the definition of Weyl invariant Jacobi forms and Wirthm\"{u}ller's structure theorem. In \S \ref{Sec:3} we introduce some basic facts about root lattice $E_8$ and Jacobi forms, as well as several techniques to construct Jacobi forms. \S \ref{Sec:4}  is devoted to our first main theorem and its proof. In \S \ref{Sec:5} we focus attention on $W(E_8)$-invariant Jacobi forms of small index. We ascertain the structure of $J^{\w ,W(E_8)}_{*, E_8,t}$ for $t=2,3,4$ and construct all generators. We also present two isomorphisms between the spaces of weak Jacobi forms for lattices of different types, which give new descriptions of $J^{\w ,W(E_8)}_{*, E_8,2}$ and $J^{\w ,W(E_8)}_{*, E_8,3}$. Besides, we develop an approach based on pull-backs of Jacobi forms to discuss the possible minimum weight of generators of $J^{\w ,W(E_8)}_{*, E_8,t}$ for $t=5$ and $6$. In \S \ref{Sec:6} one application is given. We estimate the dimension of the space of modular forms for the orthogonal group $\Orth^{+}(2U\oplus E_8(-1))$ and give it an upper bound using our theory of $W(E_8)$-invariant Jacobi forms.  This upper bound almost coincides with the exact dimension obtained by \cite{HU}.

We close the introduction with many general notations.  Let $\ZZ$ and $\NN$ denote the sets of integers and non-negative integers, respectively. Let $(\cdot , \cdot)$ denote the standard scalar product on $\RR^n$.  The ring of $\SL_2(\ZZ)$ modular forms is always denoted by $M_*$. The function $\Delta=\eta^{24}$ denotes the holomorphic cusp form of weight $12$ on $\SL_2(\ZZ)$. We denote the Eisenstein series of weights $4$ and $6$ on $\SL_2(\ZZ)$ by $E_4$ and $E_6$, respectively.  The symbol $E_8$ always stands for the root system of type $E_8$ and we use $E_4^2$ to represent the Eisenstein series of weight $8$.

\section{Background: Weyl invariant Jacobi forms}\label{Sec:2}
In this section we define the Weyl invariant Jacobi forms and recall Wirthm\"{u}ller's structure theorem.  Let $R$ be an irreducible root system of rank $r$. Then $R$ is one of the following types (see \cite{B})
\begin{align*}
&A_n (n\geq 1),& &B_n (n\geq 2),& &C_n (n\geq 2),& &D_n (n\geq 3),& &E_6,& &E_7,& &E_8,& &G_2,& &F_4.&
\end{align*}
Let $L(R)$ be the root lattice generated by $R$. When $L(R)$ is an odd lattice, we equip $L(R)$ with the new bilinear form $2(\cdot,\cdot) $ rescaled by 2. We denote the normalized bilinear form of $L(R)$ by $\latt{\cdot,\cdot}$.  In what follows, let $L(R)^*$ denote the dual lattice of $L(R)$ and $W(R)$ denote the Weyl group of $R$. As $L(R)$ are now even positive-definite lattices, we can define $W(R)$-invariant Jacobi forms with respect to $L(R)$ in the following way.

\begin{definition}\label{def}
Let $\varphi : \HH \times (L(R) \otimes \CC) \rightarrow \CC$ be a holomorphic function and $k\in \ZZ$, $t\in \NN$. If $\varphi$ satisfies the following properties
\begin{itemize}
\item Weyl invariance:
\begin{equation}\label{def1}
\varphi(\tau, \sigma(\mathfrak{z}))=\varphi(\tau, \mathfrak{z}), \quad \sigma\in W(R),
\end{equation}
\item Quasi-periodicity:
\begin{equation}
\varphi (\tau, \mathfrak{z}+ x \tau + y)= \exp\left(-t\pi i [ \latt{x,x}\tau +2\latt{x,\mathfrak{z}} ]\right) \varphi ( \tau, \mathfrak{z} ), \quad x,y\in L(R),
\end{equation}
\item Modularity:
\begin{equation}
\varphi \left( \frac{a\tau +b}{c\tau + d},\frac{\mathfrak{z}}{c\tau + d} \right) = (c\tau + d)^k \exp\left( t\pi i \frac{c\latt{\mathfrak{z},\mathfrak{z}}}{c \tau + d}\right) \varphi ( \tau, \mathfrak{z} ), \quad \left( \begin{array}{cc}
a & b \\ 
c & d
\end{array} \right)   \in \SL_2(\ZZ),
\end{equation}
\item $\varphi ( \tau, \mathfrak{z} )$ admits a Fourier expansion of the form
\begin{equation}\label{def4}
\varphi ( \tau, \mathfrak{z} )= \sum_{ n\in \NN }\sum_{ \ell \in L(R)^*}f(n,\ell)e^{2\pi i (n\tau + \latt{\ell,\mathfrak{z}})},
\end{equation}
\end{itemize}
then $\varphi$ is called a $W(R)$-invariant weak Jacobi form of weight $k$ and index $t$. If for any $2nt - \latt{\ell,\ell} <0$ we have $f(n,\ell)= 0$, then $\varphi$ is called a $W(R)$-invariant holomorphic Jacobi form.  If for any $2nt - \latt{\ell,\ell} \leq 0$ we have $f(n,\ell)= 0$, then $\varphi$ is called a $W(R)$-invariant Jacobi cusp form.
We denote by 
$$
J^{\w ,W(R)}_{k,L(R),t} \supsetneq  J^{W(R)}_{k,L(R),t} \supsetneq J^{\cusp ,W(R)}_{k,L(R),t}
$$
the vector spaces of $W(R)$-invariant weak, holomorphic and cusp Jacobi forms of weight $k$ and index $t$, respectively.
\end{definition}

We next introduce many notations about root systems following \cite{B}. The dual root system of $R$ is defined as 
\begin{equation}
R^\vee=\{ r^\vee: r\in R \},
\end{equation}
where $r^\vee=\frac{2}{(r,r)}r$ is the coroot of $r$.
The weight lattice of $R$ is defined as
\begin{equation}
\Lambda(R)=\left\{ v\in R\otimes \QQ: (r^\vee, v)\in \ZZ,\; \forall\; r \in R\right\}.
\end{equation}
Let $\widetilde{\alpha}$ denote the highest root of $R^\vee$. In 1992,  Wirthm\"{u}ller  proved the following theorem.

\begin{theorem}[see Theorem 3.6 in \cite{Wi}]\label{thwi}
If $R$ is not of type $E_8$, then the bigraded ring of $W(R)$-invariant weak Jacobi forms over the ring of $\SL_2(\ZZ)$ modular forms is the polynomial algebra in $r+1$ basic $W(R)$-invariant weak Jacobi forms of weight $-k(j)$ and index $m(j)$
\begin{equation*}
\varphi_{-k(j),m(j)}(\tau,\mathfrak{z}), \quad j=0, 1,...,r.
\end{equation*}
Apart from $k(0)=0$ and $m(0)=1$, the indices $m(j)$ are the coefficients of $\widetilde{\alpha}^\vee$  written as a linear combination of the simple roots of $R$. The integers $k(j)$ are the degrees of the generators of the ring of $W(R)$-invariant polynomials and also the exponents of the Weyl group $W(R)$ increased by $1$.
\end{theorem}
 
We formulate the weights and indices of these generators in Table \ref{tablewi}. For a lattice $L$, we write $\Orth(L)$ for the integral orthogonal group of $L$. It is known that $A_3\cong D_3$. Since $C_2$ is isomorphic to $B_2$ via scaling by $\sqrt{2}$ and a $45$ degree rotation, the spaces of Jacobi forms for them are isomorphic.   Note that $W(C_n)/W(D_n)=\ZZ/2\ZZ$ and $W(C_n)=\Orth(D_n)$ if $n\neq 4$. Remark that the $W(A_1)$-invariant Jacobi forms are the classical Jacobi forms in the sense of Eichler and Zagier \cite{EZ}. The generators for root systems of types $A_n$, $B_n$ and $D_4$ were constructed in \cite{Ber99}. The generators for root systems $E_6$ and $E_7$ can be found in \cite{Sa2, Sat98}.

\begin{table}[ht]
\caption{Weights and indices of generators of Weyl invariant weak Jacobi forms ($B_n: n\geq 2$, $C_n: n\geq 3$, $D_n: n\geq 4$)}\label{tablewi}
\renewcommand\arraystretch{1.5}
\noindent\[
\begin{array}{|c|c|c|c|}
\hline 
R & L(R) & W(R) & (k(j),m(j)) \\ 
\hline 
A_n & A_n & W(A_n) & (0,1), (j,1) : 2\leq j\leq n+1\\ 
\hline 
B_n & nA_1 & \Orth(nA_1) & (2j,1) : 0\leq j \leq n  \\ 
\hline 
C_n & D_n & W(C_n) & (0,1), (2,1), (4,1), (2j,2) : 3\leq j \leq n  \\ 
\hline 
D_n & D_n & W(D_n) &  (0,1), (2,1), (4,1), (n,1), (2j,2) : 3\leq j \leq n-1 \\ 
\hline
E_6 & E_6 & W(E_6) & (0,1), (2,1), (5,1), (6,2), (8,2), (9,2), (12,
3)  \\ 
\hline
E_7 & E_7 & W(E_7) & (0,1), (2,1), (6,2), (8,2), (10,2), (12,
3), (14,3), (18,4)  \\ 
\hline
G_2 & A_2 & \Orth(A_2) & (0,1), (2,1), (6,2)  \\ 
\hline
F_4 & D_4 & \Orth(D_4) &  (0,1), (2,1), (6,2), (8,2), (12,3) \\ 
\hline
\end{array} 
\]
\end{table}

Wirthm\"{u}ller's theorem does not cover the case $R=E_8$. In the next sections, we will focus on $W(E_8)$-invariant Jacobi forms and present a proper extension of Wirthm\"{u}ller's theorem to this case.

\section{ \texorpdfstring{$W(E_8)$}{W(E8)}-invariant Jacobi forms}\label{Sec:3}

In this section we give a brief overview of root lattice $E_8$ and introduce many useful facts about $W(E_8)$-invariant Jacobi forms.

\subsection{Notations and basic properties} 

We first introduce the Jacobi theta functions. Let $q=e^{2\pi i \tau}$ and $\zeta=e^{2\pi i z}$, where $\tau \in \HH$ and $z\in \CC$. The Jacobi theta functions of level two
(see \cite[Chapter 1]{Mu}) are defined as 
\begin{align*}
&\vartheta_{00}(\tau,z)=
\sum_{n\in \ZZ}q^{\frac{n^2}{2}}\zeta^n,& 
&\vartheta_{01}(\tau,z)=\sum_{n\in \ZZ}(-1)^nq^{\frac{n^2}{2}}\zeta^n,\\
&\vartheta_{10}(\tau,z)=q^{\frac 1{8}}\zeta^{\frac 1{2}}
\sum_{n\in \ZZ}q^{\frac{n(n+1)}{2}}\zeta^{n},&  &
\vartheta_{11}(\tau,z)=iq^{\frac 1{8}}\zeta^{\frac 1{2}}
\sum_{n\in \ZZ}(-1)^nq^{\frac{n(n+1)}{2}}\zeta^{n}.
\end{align*}
The Jacobi triple product formula
\begin{equation}
\vartheta(\tau,z) = 
 -q^{\frac{1}{8}} \zeta^{-\frac{1}{2}} \prod_{n= 1}^\infty (1-q^{n-1}\zeta)(1-q^n \zeta^{-1})(1-q^n) 
\end{equation}
defines a holomorphic Jacobi form of  weight $\frac{1}{2}$ and index $\frac{1}{2}$
with a multiplier system of order $8$ (see \cite{GN}). It plays an important role in the constructions of basic Jacobi forms. Note that $\vartheta(\tau,z) =-i\vartheta_{11}(\tau,z)$ and $\ZZ\tau+\ZZ$ is the set of its simple zeros. 

We next recall some standard facts about root lattice $E_8$. For a more careful treatment of this important lattice, we refer to \cite{B, SC}.  The lattice $E_8$ is the unique positive-definite even unimodular lattice of rank 8 and one of its constructions is as follows 
$$\left\{ (x_1, \dots , x_8) \in \frac{1}{2}\ZZ^8 : x_1\equiv \cdots \equiv x_8 \m 1, \, x_1+ \cdots + x_8 \equiv 0 \m  2\right\}. $$
The following eight vectors
\begin{align*}
&\alpha_1=\frac{1}{2}(1,-1,-1,-1,-1,-1,-1,1)& &\alpha_2=(1,1,0,0,0,0,0,0)&\\
&\alpha_3=(-1,1,0,0,0,0,0,0)& &\alpha_4=(0,-1,1,0,0,0,0,0)&\\
&\alpha_5=(0,0,-1,1,0,0,0,0)& &\alpha_6=(0,0,0,-1,1,0,0,0)&\\
&\alpha_7=(0,0,0,0,-1,1,0,0)& &\alpha_8=(0,0,0,0,0,-1,1,0)&
\end{align*}
are the simple roots of $E_8$ and the vectors
\begin{align*}
&w_1=(0,0,0,0,0,0,0,2)& &w_2=\frac{1}{2}(1,1,1,1,1,1,1,5)\\
&w_3=\frac{1}{2}(-1,1,1,1,1,1,1,7)& &w_4=(0,0,1,1,1,1,1,5)\\
&w_5=(0,0,0,1,1,1,1,4)& &w_6=(0,0,0,0,1,1,1,3)\\
&w_7=(0,0,0,0,0,1,1,2)& &w_8=(0,0,0,0,0,0,1,1)
\end{align*}
are the fundamental weights of $E_8$. The fundamental weights $w_j$ form the dual basis, so $(\alpha_i, w_j) = \delta_{ij}$. We remark that the highest root $\alpha_{E_8}$ of $E_8$ is $w_8$, which can be written as a linear combination of the simple roots
\begin{equation}
\alpha_{E_8}=w_8=2\alpha_1+3\alpha_2+4\alpha_3+6\alpha_4+5\alpha_5+4\alpha_6+
3\alpha_7+2\alpha_8.
\end{equation}
The exponents of the Weyl group $W(E_8)$ are 1, 7, 11, 13, 17, 19, 23, 29.
In Figure \ref{figE8} we give the extended Coxeter-Dynkin diagram of $E_8$.

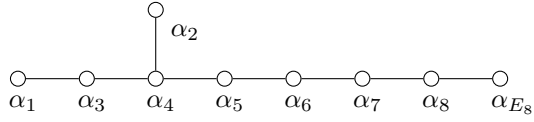
\begin{figure}[tb]
\noindent\[
\begin{tikzpicture}
\begin{scope}[start chain]
\foreach \dyni in {1,3,4,5,6,7,8,E_8} {
\dnode{\dyni}
}
\end{scope}
\begin{scope}[start chain=br going above]
\chainin (chain-3);
\dnodebr{2}
\end{scope}
\end{tikzpicture}
\]
\caption{Extended Coxeter-Dynkin diagram of $E_8$} 
\label{figE8}
\end{figure}

By \cite{SC}, the Weyl group $W(E_8)$ is of order $2^{14}3^55^27=696729600$ and it is generated by all permutations of $8$ letters, all even sign changes, and the matrix $\operatorname{diag}\{H_4, H_4\}$, where $H_4$ is the Hadamard matrix
$$H_4=\frac{1}{2}\left( \begin{array}{rrrr}
1 & 1 & 1 & 1 \\ 
1 & -1 & 1 & -1 \\ 
1 & 1 & -1 & -1 \\ 
1 & -1 & -1 & 1
\end{array}  \right).$$

In the sequel, we introduce many basic properties of $W(E_8)$-invariant Jacobi forms. We first remark that $W(E_8)$-invariant weak Jacobi forms are all of even weight  because the operator $\mathfrak{z} \mapsto -\mathfrak{z}$ belongs to $W(E_8)$. In view of the singular weight (see \cite{G} or \cite{G94}), we have
\begin{equation}
J_{k,E_8,t}^{W(E_8)}=\{0\} \quad \text{if}\quad  k<4,
\end{equation}
and if $\phi \in J_{4,E_8,t}^{W(E_8)}$ then for its non-zero Fourier coefficients $f(n,\ell)$ we have $2nt-(\ell,\ell)=0$. The two facts are also true for holomorphic Jacobi forms with a character.

The following fact is very standard. For a proof, we refer to \cite[Lemma 2.1]{G94}.

\begin{lemma}\label{lem2.3}
Let $\varphi \in J^{\w ,W(E_8)}_{k,E_8,t}$. Then the coefficients $f(n,\ell)$ of $\varphi$ depend only on the class of $\ell$ in $E_8/tE_8$ and on the value of $2nt-(\ell,\ell)$. Besides, if $f(n,\ell)\neq 0$ then we have
$$ 2nt- (\ell,\ell) \geq - \min\{(v,v): v\in \ell+ tE_8 \}.$$
\end{lemma}

Using the same technique as in the proof of \cite[Theorem 8.4]{EZ}, we  prove the following lemma.

\begin{lemma}\label{lem: module}
Let $t\in \NN$. The space $J^{\w ,W(E_8)}_{*, E_8,t}$ (resp. $J^{W(E_8)}_{*, E_8,t}$, $J^{\cusp ,W(E_8)}_{*, E_8,t}$) of $W(E_8)$-invariant weak Jacobi forms (resp. holomorphic Jacobi forms, Jacobi cusp forms) of index $t$ is a free module over $M_*$. Besides, the three modules have the same rank over $M_*$.
\end{lemma}

We next study carefully the Fourier expansion of $W(E_8)$-invariant Jacobi forms. 
For any $n\in \NN$, we define the $q^n$-term of $\varphi$ as 
$$
[\varphi]_{q^n}=\sum_{\ell\in E_8}f(n,\ell)e^{2\pi i (\ell,\mathfrak{z})},
$$
and for any $m\in E_8$ we denote the Weyl orbit of $m$ by
\begin{equation}
\orb(m)=\sum_{\sigma\in W(E_8)/W(E_8)_{m}}e^{2\pi i (\sigma(m),\mathfrak{z})},
\end{equation}
where $W(E_8)_{m}$ is the stabilizer subgroup of $W(E_8)$ with respect to $m$.

We see from Lemma \ref{lem2.3} that $[\varphi]_{q^n}$ is an exponential polynomial invariant under $W(E_8)$.  
By \cite[Th\'{e}or\`{e}me VI.3.1]{B} or \cite[Theorem 3.6.1]{L}, we get the following lemma.

\begin{lemma}\label{lem0}
Let $\varphi \in J^{\w ,W(E_8)}_{k,E_8,t}$ and $n\in \NN$. We have
$$ [\varphi]_{q^n} \in \CC[\,\orb(w_j):1\leq j \leq 8].$$
Moreover,  $\orb(w_j)$, $1\leq j \leq 8$, are algebraically independent over $\CC$.
\end{lemma}

\subsection{Constructions of Jacobi forms}
In this subsection we introduce some techniques to construct Jacobi forms.  We first recall the following weight raising differential operators, which play a crucial role in this paper. Such operators can also be found in \cite{CK} for the general case or in \cite{EZ} for classical Jacobi forms.
\begin{lemma}\label{diffoperator}
Let $\varphi(\tau,\mathfrak{z})=\sum f(n,\ell)e^{2\pi i (n\tau + (\ell,\mathfrak{z}))}$ be a $W(E_8)$-invariant weak Jacobi form of weight $k$ and index $t$. Then $H_{k}(\varphi)$ is a $W(E_8)$-invariant weak Jacobi form of weight $k+2$ and index $t$, where 
\begin{align*}
H_k(\varphi)(\tau,\mathfrak{z})&=H(\varphi)(\tau,\mathfrak{z})+\frac{4-k}{12}E_2(\tau)\varphi(\tau,\mathfrak{z}),\\
H(\varphi)(\tau,\mathfrak{z})&=\sum_{n\in \NN}\sum_{\ell\in E_8} \left[n-\frac{1}{2t}(\ell,\ell) \right]f(n,\ell)e^{2\pi i (n\tau + (\ell,\mathfrak{z}))},
\end{align*}
and $E_2(\tau)=1-24\sum_{n\geq 1}\sigma(n)q^n$ is the Eisenstein series of weight $2$ on $\SL_2(\ZZ)$.
\end{lemma}

\begin{proof}
Let $\mathfrak{z}=\sum_{j=1}^8z_j\alpha_j \in E_8\otimes \CC$, $z_j\in\CC$. We define $\frac{\partial}{\partial\mathfrak{z}}=\sum_{j=1}^8 w_j \frac{\partial}{\partial z_j}$. Then 
$$
\left(\frac{\partial}{\partial\mathfrak{z}},\frac{\partial}{\partial\mathfrak{z}}\right)e^{2\pi i(\ell,\mathfrak{z})}=-4\pi^2(\ell,\ell)e^{2\pi i(\ell,\mathfrak{z})}
$$
and the operator $H(\cdot)$ is equal to the heat operator
$$
H=\frac{1}{2\pi i}\frac{\partial}{\partial\tau}+\frac{1}{8\pi^2 t}\left(\frac{\partial}{\partial\mathfrak{z}},\frac{\partial}{\partial\mathfrak{z}}\right).
$$
The transformations of $H(\varphi)$ with respect to the actions of $\SL_2(\ZZ)$ and the Heisenberg group of $E_8$ were calculated in \cite[Lemma 3.3]{CK}. From these transformations and the transformation of $E_2$ under $\SL_2(\ZZ)$,  we can show that $H_k(\varphi)$ is indeed invariant under $\SL_2(\ZZ)$ and the Heisenberg group. Therefore, it is a $W(E_8)$-invariant weak Jacobi form of weight $k+2$ and index $t$.
\end{proof}

The next lemma gives a quite useful identity related to the $q^0$-term of any weak Jacobi form of weight zero. It is a particular case of \cite[Proposition 2.6]{G18}. 

\begin{lemma}\label{lf}
Let $\varphi(\tau,\mathfrak{z})=\sum f(n,\ell)e^{2\pi i (n\tau + (\ell,\mathfrak{z}))}$ be a $W(E_8)$-invariant weak Jacobi form of weight $0$ and index $t$. Then we have the following identity
\begin{equation*}
2t\sum_{\ell\in E_8}f(0,\ell)=3\sum_{\ell\in E_8}f(0,\ell)(\ell,\ell).
\end{equation*}
\end{lemma}
\begin{proof}
By Lemma \ref{diffoperator}, we have $H_0(\varphi)\in J^{\w ,W(E_8)}_{2,E_8,t}$. It follows that $H_0(\varphi)(\tau,0)=0$ because it is a holomorphic modular form of weight 2 on $\SL_2(\ZZ)$. Therefore $H_0(\varphi)(\tau,0)$ has zero constant term, which establishes the desired identity.
\end{proof}

The following index raising operator can be found in \cite[Corollary 1]{G}.
\begin{lemma}\label{lem: index}
Let $s$ be a positive integer and $\varphi \in J^{\w ,W(E_8)}_{k,E_8,t}$. Then we have
$$ \varphi\lvert_k T_{-}(s)=s^{-1}\sum_{\substack{ ad=s\\ b\m d }}a^{k}\varphi\left(\frac{a\tau+b}{d},a\mathfrak{z}\right) \in J^{\w ,W(E_8)}_{k,E_8,st}.$$
Moreover, the function $\varphi\lvert_k T_{-}(s)$ has a Fourier expansion of the form 
$$\left(\varphi\lvert_k T_{-}(s)\right)(\tau,\mathfrak{z})=\sum_{\substack{n\in \NN\\ \ell\in E_8}}\sum_{\substack{d\in \NN\\ d\vert (n,\ell,s)}}d^{k-1}f\left(\frac{ns}{d^2},\frac{\ell}{d}\right)e^{2\pi i (n\tau + (\ell,\mathfrak{z}))},$$
where $f(n,\ell)$ are the Fourier coefficients of $\varphi$ and the notation $d\vert (n,\ell,s)$ means that $d\vert (n,s)$ and $d^{-1}\ell \in E_8$.
\end{lemma}

Since weight $4$ is the singular weight, it is impossible to construct $W(E_8)$-invariant holomorphic Jacobi forms of weight $6$ from Jacobi forms of weight $4$ by the differential operators introduced in Lemma \ref{diffoperator}. Next, we give a resultful method to construct $W(E_8)$-invariant holomorphic Jacobi forms of weight $6$. It is well-known that the theta function of the root lattice $E_8$ defined as
\begin{equation}
\begin{split}
\vartheta_{E_8}(\tau,\mathfrak{z})=&\sum_{\ell \in E_8}\exp\left(\pi i (\ell,\ell)\tau+2\pi i (\ell, \mathfrak{z})\right)\\
  =&\frac{1}{2}\left[\prod_{j=1}^8\vartheta(\tau,z_j) +\prod_{j=1}^8\vartheta_{00}(\tau,z_j)+\prod_{j=1}^8\vartheta_{01}(\tau,z_j)+\prod_{j=1}^8\vartheta_{10}(\tau,z_j) \right]
\end{split}
\end{equation}
is a $W(E_8)$-invariant holomorphic Jacobi form of weight $4$ and index $1$. One can check that $\vartheta_{E_8}(t\tau,t\mathfrak{z})$ is a $W(E_8)$-invariant holomorphic Jacobi form of weight $4$ and index $t$ for the congruence subgroup $\Gamma_0(t)$. We take a modular form of weight $2$ on $\Gamma_0(t)$ and note it by $g(\tau)$. Then $g(\tau)\vartheta_{E_8}(t\tau,t\mathfrak{z})$ is a $W(E_8)$-invariant holomorphic Jacobi form of weight $6$ and index $t$ for $\Gamma_0(t)$. Therefore, the trace operator of $g(\tau)\vartheta_{E_8}(t\tau,t\mathfrak{z})$ defined by
\begin{equation}
Tr_{\SL_2(\ZZ)}(g(\tau)\vartheta_{E_8}(t\tau,t\mathfrak{z}))=\sum_{\gamma\in \Gamma_0(t)\backslash \SL_2(\ZZ)}(g(\tau)\vartheta_{E_8}(t\tau,t\mathfrak{z}))\lvert_{6,t}\gamma
\end{equation}
is a $W(E_8)$-invariant holomorphic Jacobi form of weight $6$ and index $t$, where $\lvert_{k,t}\gamma$ is the slash action of $\gamma\in \SL_2(\ZZ)$ defined as
$$ 
(\phi \lvert_{k,t}\gamma)(\tau,\mathfrak{z})=(c\tau + d)^{-k} \exp\left(- t\pi i \frac{c(\mathfrak{z},\mathfrak{z})}{c \tau + d}\right) \phi \left( \frac{a\tau +b}{c\tau + d},\frac{\mathfrak{z}}{c\tau + d} \right). 
$$

By the index raising operators introduced in Lemma \ref{lem: index}, one can construct a $W(E_8)$-invariant holomorphic Jacobi form of weight $4$ and index $t\geq 1$:
\begin{equation}\label{X}
X_t(\tau,\mathfrak{z})=*(\vartheta_{E_8}\lvert_4T_{-}(t))(\tau,\mathfrak{z})=1+O(q) \in J^{W(E_8)}_{4,E_8,t},
\end{equation}
where $*$ is a constant such that $X_t(\tau,0)=E_4(\tau)$.  Sakai  \cite{Sa1} constructed five $W(E_8)$-invariant holomorphic Jacobi forms of weight $4$ as
\begin{align}
&A_j(\tau,\mathfrak{z})=X_j(\tau,\mathfrak{z}),\;j=1,2,3,5,\\
& A_4(\tau,\mathfrak{z})=A_1(\tau,2\mathfrak{z}).
\end{align}

Since $E_2(\tau)-pE_2(p\tau) \in M_2(\Gamma_0(p))$ if $p$ is a prime number, one can construct $W(E_8)$-invariant holomorphic Jacobi forms of weight $6$ and prime index. Furthermore, one can construct $W(E_8)$-invariant holomorphic Jacobi forms of weight $6$ and any index bigger than $1$ using the index raising operators. These Jacobi forms may reduce to Eisenstein series $E_6(\tau)$ by taking $\mathfrak{z}=0$.  Sakai \cite[Appendix A.1]{Sa1} also constructed $W(E_8)$-invariant holomorphic Jacobi forms of weight $6$ and index $2$, $3$, $4$, $6$ by choosing particular modular forms of weight $2$ on congruence subgroups. More precisely, they were constructed in the following way
\begin{align}
B_2(\tau,\mathfrak{z})&=*Tr_{\SL_2(\ZZ)}\left[(2E_2(2\tau)-E_2(\tau))\vartheta_{E_8}(2\tau,2\mathfrak{z})\right],\\
B_3(\tau,\mathfrak{z})&=*Tr_{\SL_2(\ZZ)}\left[(3E_2(3\tau)-E_2(\tau))\vartheta_{E_8}(3\tau,3\mathfrak{z})\right],\\
B_4(\tau,\mathfrak{z})&=*Tr_{\SL_2(\ZZ)}\left[\vartheta_{01}^4(2\tau)\vartheta_{E_8}(4\tau,4\mathfrak{z})\right],\\
B_6(\tau,\mathfrak{z})&=*Tr_{\SL_2(\ZZ)}\left[(3E_2(3\tau)-E_2(\tau))\vartheta_{E_8}(6\tau,6\mathfrak{z})\right],
\end{align}
here, these constants $*$ are chosen such that $B_j(\tau,0)=E_6(\tau)$.

\subsection{Lifting elliptic modular forms to Jacobi forms}
In this subsection we give another way to construct $W(E_8)$-invariant Jacobi forms.
In \cite{Sch1, Sch2}, Scheithauer constructed a map which lifts scalar-valued modular forms on congruence subgroups to modular forms for the Weil representation. In view of the isomorphism between modular forms for the Weil representation and Jacobi forms, we can easily build a lifting  from scalar-valued modular forms on congruence subgroups to Jacobi forms. 

For our purpose, we focus on the lattices $E_8(p)$, which is the group $E_8$ equipped with the following rescaled bilinear form 
$$\latt{\cdot,\cdot}_p= p(\cdot,\cdot),$$ 
where $p$ is a prime number. Let $D(p)=E_8(p)^*/E_8(p)$ be the discriminant group of $E_8(p)$. Then $D(p)$ is of level $p$. Let $\{e_\gamma: \gamma \in D(p)\}$ be the basis of the group ring $\CC[D(p)]$. We denote the Weil representation of $\SL_2(\ZZ)$ on $\CC[D(p)]$ by $\rho_{D(p)}$ and the orthogonal group of $D(p)$ by $\Orth(D(p))$. Let $M^{inv}_k(\rho_{D(p)})$ be the space of holomorphic modular forms for $\rho_{D(p)}$ of weight $k$ which are invariant under the action of $\Orth(D(p))$. By \cite[Theorem 6.2]{Sch1}, we have the following proposition.

\begin{proposition}\label{propsch}
Let $f\in M_{k}(\Gamma_0(p))$ be a scalar-valued holomorphic modular form for $\Gamma_0(p)$ of weight $k$. Then 
\begin{equation}
F_{\Gamma_0(p),f,0}(\tau):= \sum_{M\in \Gamma_0(p) \backslash \SL_2(\ZZ)} f\lvert_M(\tau) \rho_{D(p)}(M^{-1})e_0 \in M^{inv}_k(\rho_{D(p)}).
\end{equation}
If we write 
$$ f\lvert_S (\tau)=\sum_{t=0}^{p-1}g_t(\tau),  \quad S=\left(\begin{array}{cc}
0 & -1 \\ 
1 & 0
\end{array}  \right),$$
where 
$$g_t(\tau+1)=\exp\left(\frac{2t\pi i}{p}\right) g_t(\tau), \quad 0\leq t \leq p-1,$$
then we have 
\begin{equation}\label{fs1}
F_{\Gamma_0(p),f,0}(\tau)=f(\tau)e_0+\frac{1}{p^3}\sum_{\gamma\in D(p)} g_{j_\gamma}(\tau) e_\gamma,
\end{equation}
here $j_\gamma/p = -\latt{\gamma,\gamma}_p/2 \mod 1$ for $\gamma\in D(p)$.

Moreover, the map 
\begin{equation}\label{maps2}
f\in M_{k}(\Gamma_0(p)) \mapsto F_{\Gamma_0(p),f,0}\in M^{inv}_k(\rho_{D(p)})
\end{equation}
is an isomorphism.
\end{proposition}

\begin{proof}
The formula (\ref{fs1}) can be proved by \cite[Theorem 6.5]{Sch1}. Since there exists $\gamma\neq 0$ in $D(p)$ with $\latt{\gamma,\gamma}_p\in 2\ZZ$, we conclude from (\ref{fs1}) that the map (\ref{maps2}) is injective. The surjectivity of (\ref{maps2}) follows from \cite[Corollary 5.5]{Sch2}.
\end{proof}

Recall that the theta function for the lattice $E_8(p)$ are defined by
\begin{equation}
\Theta_{\gamma}^{E_8(p)}(\tau,\mathfrak{z})=\sum_{\ell \in \gamma+E_8}\exp\left(\pi i\latt{\ell,\ell}_p \tau + 2\pi i\latt{\ell,\mathfrak{z}}_p \right), \quad \gamma\in D(p).
\end{equation}

By means of \cite[Lemma 2.3]{G94}, we obtain the following result.

\begin{proposition}\label{Prop:lift}
Under the assumptions of Proposition \ref{propsch}, if we write 
$$F_{\Gamma_0(p),f,0}(\tau)=\sum_{\gamma\in D(p)} F_{\Gamma_0(p),f,0;\gamma}(\tau)e_\gamma,$$ then the function
\begin{equation}
\Phi_{\Gamma_0(p),f,0}(\tau,\mathfrak{z})=\sum_{\gamma\in D(p)} F_{\Gamma_0(p),f,0;\gamma}(\tau)\Theta_{\gamma}^{E_8(p)}(\tau,\mathfrak{z})
\end{equation}
is a $W(E_8)$-invariant holomorphic Jacobi form of weight $k+4$ and index $p$. Moreover, the map sends cusp forms to Jacobi cusp forms.
\end{proposition}

As an application of the above map, we can construct $W(E_8)$-invariant Jacobi forms of small weight. For example,
\begin{align}
\Phi_{\Gamma_0(p),1,0}(\tau,\mathfrak{z})&\in J_{4,E_8,p}^{W(E_8)},\\
\Phi_{\Gamma_0(p),pE_2(p\tau)-E_2(\tau),0}(\tau,\mathfrak{z})&\in J_{6,E_8,p}^{W(E_8)}.
\end{align}

We remark that the map $f\mapsto \Phi_{\Gamma_0(p),f,0}$ is injective. But it is not surjective in general unless the homomorphism $\Orth(E_8(p))=W(E_8) \to \Orth(D(p))$ is surjective. 
The map $\Orth(E_8(p)) \to \Orth(D(p))$ is surjective if and only if $p=2$. Therefore, as an analogue of the natural isomorphism
\begin{align*}
M_{k}(\SL_2(\ZZ)) &\longrightarrow J_{k+4,E_8,1}^{W(E_8)}\\
f(\tau)&\longmapsto f(\tau)\vartheta_{E_8}(\tau,\mathfrak{z}),
\end{align*}
we can build the following isomorphism.

\begin{corollary}\label{coro:index2}
We have the following isomorphism
\begin{align*}
M_k(\Gamma_0(2)) &\longrightarrow J_{k+4,E_8,2}^{W(E_8)}\\
f(\tau) &\longmapsto  \Phi_{\Gamma_0(2),f,0}(\tau,\mathfrak{z})
\end{align*}
and it induces an isomorphism between the subspaces of cusp forms.
\end{corollary}

\section{The ring of \texorpdfstring{$W(E_8)$}{W(E8)}-invariant Jacobi forms}\label{Sec:4}
In this section we study the bigraded ring of $W(E_8)$-invariant weak Jacobi forms.
Lemma \ref{lem: module} says that the space of $W(E_8)$-invariant weak Jacobi forms of fixed index is a free module over $M_*$. In what follows, we give an explicit formula to compute the number of generators. The following is our main theorem, which is  an explicit version of the Sakai conjecture in \cite[Page 55]{Sa2}.

\begin{theorem}\label{MTH}
Let $t$ be a positive integer. The space $J_{*,E_8,t}^{\w ,W(E_8)}$ of $W(E_8)$-invariant weak Jacobi forms of index $t$ is a free module of rank $r(t)$ over $M_*$, where $r(t)$ is given by the generating series
\begin{equation}\label{eq1}
\frac{1}{(1-x)(1-x^2)^2(1-x^3)^2(1-x^4)^2(1-x^5)(1-x^6)}=\sum_{t\geq 0}r(t)x^t.
\end{equation}
Equivalently, the bigraded algebra $J_{*,E_8,*}^{\w ,W(E_8)}$ of $W(E_8)$-invariant weak Jacobi forms is contained in the polynomial algebra in nine variables over the fractional field of $\CC[E_4,E_6]$. More precisely, 
\begin{equation}
J_{*,E_8,*}^{\w ,W(E_8)} \subsetneq \CC\left( E_4, E_6 \right) \left[A_1,A_2,A_3,A_4,A_5,B_2,B_3,B_4,B_6 \right],
\end{equation}
and the functions $A_1$, $A_2$, $A_3$, $A_4$, $A_5$, $B_2$, $B_3$, $B_4$, $B_6$ are algebraically independent over $M_*$.
\end{theorem}

\begin{proof}
We denote the rank of $J_{*,E_8,t}^{\w ,W(E_8)}$ over $M_*$ by $R(t)$. It is sufficient to show $R(t)=r(t)$ in order to prove the theorem. 

On the one hand, the nine Jacobi forms $A_i$ and $B_j$  are algebraically independent over $M_*$. Sakai constructed  nine meromorphic Jacobi forms $\alpha_{k,t}$ in \cite[Appendix A.2]{Sa1} and \cite[\S 3.2 and page 74]{Sa2} whose $q^0$-terms involve the eight fundamental Weyl orbits $\orb(w_i)$. For each $\alpha_{k,t}$, we multiply it by a certain power of $E_4$ to cancel the pole. In this way, we get a weak Jacobi form $\widehat{\alpha}_{k,t}$ whose $q^0$-term is the same as that of $\alpha_{k,t}$. By considering the polynomial combinations of these $\widehat{\alpha}_{k,t}$ over $M_*$, we can construct nine basic $W(E_8)$-invariant weak Jacobi forms whose $q^0$-term is a single fundamental Weyl orbit: one Jacobi form of index 1 with $q^0$-term $1$; two Jacobi forms of index $2$ with $q^0$-terms $\orb(w_1)$ and $\orb(w_8)$ respectively; two Jacobi forms of index $3$ with $q^0$-terms $\orb(w_2)$ and $\orb(w_7)$ respectively; two Jacobi forms of index $4$ with $q^0$-terms $\orb(w_3)$ and $\orb(w_6)$ respectively; one Jacobi form of index $5$ with $q^0$-term $\orb(w_5)$; one Jacobi form of index $6$ with $q^0$-term $\orb(w_4)$. For example, the index one and index two Jacobi forms can be constructed as follows:
\begin{align*}
-\frac{1}{4}E_4\alpha_{0,1}&=1+O(q) \in J_{4,E_8,1}^{W(E_8)},\\
-\frac{9}{2}E_6E_4^2\alpha_{-6,2}-\frac{9}{8}E_6^2E_4\alpha_{-8,2}+135E_4^2\alpha_{0,1}^2&=\orb(w_1)+O(q)\in J_{8,E_8,2}^{\w, W(E_8)},\\
-\frac{1}{4}E_6E_4^2\alpha_{-6,2}+\frac{1}{48}E_6^2E_4\alpha_{-8,2}+15E_4^2\alpha_{0,1}^2&=\orb(w_8)+O(q)\in J_{8,E_8,2}^{\w, W(E_8)}.
\end{align*}
We note that $\widehat{\alpha}_{0,1}=E_4\alpha_{0,1}$, $\widehat{\alpha}_{-6,2}=E_4^2\alpha_{-6,2}$ and $\widehat{\alpha}_{-8,2}=E_4\alpha_{-8,2}$.
Since $\orb(w_i)$, $1\leq i \leq 8$, are algebraically independent over $\CC$, it is easy to show that the nine basic weak Jacobi forms are algebraically independent over $M_*$. We notice that the nine basic weak Jacobi forms are in fact polynomials in $A_i$ and $B_j$ over $M_*[\frac{1}{\Delta}]$. Therefore, the nine Jacobi forms $A_i$ and $B_j$ are also algebraically independent over $M_*$. This fact yields $R(t)\geq r(t)$. 

On the other hand, Lemma \ref{lem2} below gives $R(t)\leq r(t)$. Then the proof is completed.
\end{proof}

The next lemma is crucial to the proof of the above theorem. It also has its own interest because the values of $W(E_8)$-invariant Jacobi forms at $q=0$ are very interesting in physics (see \cite{Sa1, Sa2}).

\begin{lemma}\label{lem1}
Assume that 
$$\phi_t(\tau,\mathfrak{z})=\sum_{n\geq 0 }\sum_{\ell\in E_8}f(n,\ell)e^{2\pi i(n\tau+(\ell,\mathfrak{z}))}$$
is a $W(E_8)$-invariant weak Jacobi form of index $t$. Then we have
\begin{equation}
\sum_{\ell\in E_8}f(0,\ell)e^{2\pi i (\ell,\mathfrak{z)}} = \sum_{\substack{X\in \NN^8\\T(X)\leq t}}c(X)\prod_{i=1}^8\orb(w_i)^{x_i},
\end{equation}
where $c(X)\in \CC$ are constants, $X=(x_1,x_2,\dots,x_8)\in \NN^8$ and 
$$
T(X)=2x_1+3x_2+4x_3+6x_4+5x_5+4x_6+3x_7+2x_8.
$$
Moreover, $\orb(w_i)$, $1\leq i \leq 8$, are algebraically independent over $\CC$.
\end{lemma}

\begin{proof}
By Lemma \ref{lem0}, we know that
$$ 
[\phi_t]_{q^0}=\sum_{\ell\in E_8}f(0,\ell)e^{2\pi i (\ell,\mathfrak{z)}} \in \CC[\orb(w_i),1\leq i \leq 8]
$$
and $\orb(w_i)$, $1\leq i \leq 8$, are algebraically independent over $\CC$. We put 
$$
\Lambda_{+}=\left\{m\in E_8: (\alpha_i,m)\geq 0, 1\leq i \leq 8 \right\}=\bigoplus_{i=1}^8\NN w_i,
$$
which is the closure of a Weyl chamber.
We recall the following standard facts:
\begin{enumerate}
\item[(a)] Every $W(E_8)$-orbit in $E_8$ meets the set $\Lambda_{+}$ in exactly one point (see \cite[Th\'{e}or\`{e}me VI.1.2(ii)]{B}).
\item[(b)] Define a partial order on $E_8$ by $m\geq m'$ if $m-m'\in\bigoplus_{i=1}^8 \RR_{+}\alpha_i$. Then $m\geq \sigma(m)$ holds for all $m\in \Lambda_{+}$ and $\sigma\in W(E_8)$  (see \cite[Prop.VI.1.18]{B}).
\item[(c)] For each $m\in \Lambda_{+}$, there are only finitely many $m'\in \Lambda_{+}$ satisfying $m\geq m'$ (see \cite[p.187]{B}).
\end{enumerate}

By the fact (a), we have 
\begin{equation}\label{*}
[\phi_t]_{q^0}=\sum_{m\in \Lambda_{+}} c(m) \orb(m).
\end{equation}

For $m\in \Lambda_{+}$ with $m=\sum_{i=1}^8x_iw_i$, $x_i\in \NN$, we define 
\begin{align*}
T(m)=&(m,w_8)\\
=&2x_1+3x_2+4x_3+6x_4+5x_5+4x_6+3x_7+2x_8.
\end{align*}
When $T(m)>t$, we have 
\begin{align*}
(m-tw_8,m-tw_8)=(m,m)-2t(m,w_8)+2t^2< (m,m).
\end{align*}
In this case, by Lemma \ref{lem2.3}, there will be a Fourier coefficient of type $q^{n_0}e^{2\pi i (m-tw_8,\mathfrak{z})}$ with $n_0<0$, which contradicts the definition of weak Jacobi forms. Thus, we obtain 
\begin{equation}\label{**}
[\phi_t]_{q^0}=\sum_{\substack{ m\in \Lambda_{+}\\ T(m)\leq t}} c(m) \orb(m), 
\end{equation}
Define 
$$f_m=\prod_{i=1}^8\orb(w_i)^{x_i}, \quad m=\sum_{i=1}^8x_iw_i.$$
By the facts (b) and (c), the product $f_m$ can be written as a finite sum 
\begin{equation}\label{***}
f_m=\orb(m) +\sum_{\substack{ m_1\in \Lambda_{+}\\ m_1<m}}c_{m,m_1}\orb(m_1).
\end{equation}
We note that $ m_1<m$ implies $T(m_1)\leq T(m)$ because $m-m_1>0$ and $w_8$ is the highest root of $E_8$, which yield
\begin{align*}
T(m)- T(m_1)=(m,w_8)-(m_1,w_8)=(m-m_1,w_8)\geq 0.
\end{align*}
We therefore establish the desired formula by equations \eqref{**}, \eqref{***} and $(c)$.
\end{proof}

\begin{lemma}\label{lem2}
The space $J_{*,E_8,t}^{\w ,W(E_8)}$ of $W(E_8)$-invariant weak Jacobi forms of index $t$ is a free module of rank $\leq r(t)$ over $M_*$, where $r(t)$ is defined by $(\ref{eq1})$.
\end{lemma}
\begin{proof}
Conversely, suppose that the rank $R(t)$ of $J_{*,E_8,t}^{\w ,W(E_8)}$ over $M_*$ is larger than $ r(t)$. We assume that $\{\psi_i: 1\leq i \leq R(t)\}$ is a basis of $J_{*,E_8,t}^{\w ,W(E_8)}$ over $M_*$ and the weight of $\psi_i$ is $a_i$. We put $a=\max\{a_i:1\leq i \leq R(t)\}$.
According to Lemma \ref{lem1}, $q^0$-terms of $W(E_8)$-invariant weak Jacobi forms of index $t$ can be written as $\CC$-linear combinations of $r(t)$ fundamental elements $f_m=\prod_{i=1}^8\orb(w_i)^{x_i}, m=\sum_{i=1}^8x_iw_i$ with $T(m)\leq t$.  Since $R(t)>r(t)$, there exists a homogenous polynomial $P\neq 0$ of degree one over $M_*$ such that 
$$P(\psi_i,1\leq i \leq R(t))=\sum_{i=1}^{R(t)}c_iE_{a+4-a_i}\psi_i=O(q)\in J_{a+4,E_8,t}^{\w ,W(E_8)},$$
where $c_i$ are constants and $E_{a+4-a_i}$ are Eisenstein series of weight $a+4-a_i$ on $\SL_2(\ZZ)$.
Hence $P(\psi_i,1\leq i \leq R(t))/\Delta \neq 0 \in J_{a-8,E_8,t}^{\w ,W(E_8)}$ and thus it is a linear combination of $\psi_i$ over $M_*$, which is impossible. 
\end{proof}

Some values of rank $r(t)$ are shown in Table \ref{tablerank}.

\begin{table}[ht]
\caption{Rank of $J_{*,E_8,t}^{\w ,W(E_8)}$ over $M_*$}\label{tablerank}
\renewcommand\arraystretch{1.5}
\noindent\[
\begin{array}{|c|c|c|c|c|c|c|c|c|c|c|c|c|c|c|}
\hline 
t & 1 & 2 & 3 & 4 & 5 & 6 & 7 & 8 & 9 & 10 & 11 & 12 & 13 & 14 \\ 
\hline 
r(t) & 1 & 3 & 5 & 10 & 15 & 27 & 39 & 63 & 90 & 135 & 187 & 270 & 364 & 505 \\ 
\hline 
\end{array} 
\]
\end{table}

\begin{remark}\label{remark:Wir}
If Wirthm\"{u}ller's theorem holds for $R=E_8$, then the weights and indices of the nine generators are as follows (see Theorem \ref{thwi})
\begin{align*}
&(0,1)& &(-2,2)& &(-8,2)& &(-12,3)& &(-14,3)& \\
&(-18,4)& &(-20,4)& &(-24,5)& &(-30,6).&
\end{align*}
Therefore, Theorem \ref{MTH} implies that the statement about the indices of generators in Wirthm\"{u}ller's theorem holds for $R=E_8$.

We note that Lemma \ref{lem1} can be extended to any irreducible root system using the following facts:
\begin{itemize}
\item The weight lattice $\Lambda(R^\vee)$ of $R^\vee$ is isomorphic to the dual lattice of $L(R)$.
\item The multiplicative invariant algebra $\ZZ[\Lambda(R^\vee)]^{W(R)}$ is a polynomial algebra over $\ZZ$: the Weyl orbits of the fundamental weights of $R^\vee$ are algebraically independent generators (see \cite[Th\'{e}or\`{e}me VI.3.1 and Exemple 1]{B}).
\item $T(l)$ can be defined as $(l,\widetilde{\alpha}^\vee)$ (see \S \ref{Sec:2} for $\widetilde{\alpha}^\vee$). If $T(l)$ is greater than the index $t$, then the norm of $l-t\widetilde{\alpha}^\vee$ is smaller than the norm of $l$.
\item $l_1 < l_2$ implies  $T(l_1)\leq T(l_2)$, for any $l_1$, $l_2$ in the dual lattice of $L(R)$.
\end{itemize}
Thus, by virtue of the analogues of Lemmas \ref{lem1}, \ref{lem2}, we can give a new proof of the fact about the indices of generators in Wirthm\"{u}ller's theorem. 

The fact about the weights of generators is related to the Taylor expansion of basic weak Jacobi forms at the point $\mathfrak{z}=0$.  We next explain it more precisely. Given an irreducible root system $R$ of rank $r$, if we could find $r+1$ basic weak Jacobi forms $\phi_j$ of expected indices whose $q^0$-terms contain the corresponding Weyl orbits of the fundamental weights, then we may show by the above arguments that these Jacobi forms are algebraically independent over $M_*$ and that for each Jacobi form $\varphi$, there exist a modular form $f$ with non-zero constant term and a non-zero polynomial $P\in M_*[X_j,0\leq j \leq r]$ such that $f\varphi=P(\phi_j,0\leq j \leq r)$. If $f$ is not constant, then $f$ vanishes at a point $\tau_0\in\HH$. We observe that the leading terms of the Taylor expansion of $\phi_j$ are a homogenous $W(R)$-invariant polynomial of degree equal to the absolute value of the weight of $\phi_j$. Therefore, if these $\phi_j$ further have the expected weights and the generators of $W(R)$-invariant polynomials appear in their leading terms of Taylor expansions, then it is possible to prove that $\phi_j(\tau_0,\mathfrak{z})$ are algebraically independent over $\CC$ using the fact that the $r$ generators of $W(R)$-invariant polynomials are algebraically independent over $\CC$, which gives a  contradiction. Then $f$ is a constant and we deduce that the ring of Jacobi forms is the polynomial algebra generated by $\phi_j$ over $M_*$. The above discussions give an explanation of why the ring of Weyl invariant weak Jacobi forms is possible to be a polynomial algebra.
\end{remark}

\begin{remark}
Our main theorem shows that every $W(E_8)$-invariant weak Jacobi form can be expressed uniquely as a polynomial in $A_i$ and $B_j$ with coefficients which are meromorphic $\SL_2(\ZZ)$ modular forms (the quotients of holomorphic modular forms). By the structure results in the next section, these meromorphic modular forms are in fact holomorphic except at infinity when the index is less than or equal to $3$. In other words, we have
$$    
J_{*,E_8,t}^{\w ,W(E_8)} \subsetneq M_*\left[\frac{1}{\Delta}\right]\left[A_1,A_2,A_3,A_4,A_5,B_2,B_3,B_4,B_6 \right], \quad t=1,2,3.
$$

But when the index is larger than $3$, it is very likely that the above meromorphic modular forms have a pole at one point $\tau_0\in \HH$, which is different from the case in \cite{EZ}. In \cite{ZGHPKL}, the authors checked numerically that the $W(E_8)$-invariant holomorphic Jacobi form of weight $16$ and index $5$ defined as
$$
P= 864A_1^3A_2+21E_6^2A_5-770E_6A_3B_2+3825A_1B_2^2-840E_6A_2B_3+60E_6A_1B_4
$$
vanishes at the zero points $\tau=\pm \frac{1}{2}+\frac{\sqrt{3}}{2}i$ of $E_4$ for general $E_8$ elliptic parameters.  If the zeros of $P$ and $E_4$ do indeed coincide, then $P/E_4$ will be a $W(E_8)$-invariant holomorphic Jacobi form of weight $12$ and index $5$.
\end{remark}

\begin{remark}
In some sense, the choice of generators $A_i$ and $B_j$ in our main theorem is optimal. By the structure theorems in the next section, it is very natural to choose $A_1$, $A_2$, $A_3$, $A_5$, $B_2$, $B_3$ as generators because the corresponding spaces are all  one-dimensional. There are $2$ independent $W(E_8)$-invariant holomorphic Jacobi forms of weight $4$ and index $4$. One is $A_4$ and the other is $X_4$ (see (\ref{X})). But $\Delta X_4$ can be expressed as a polynomial in our generators $A_1$, $A_2$, $A_3$, $A_4$, $B_2$, $B_3$ and Eisenstein series $E_4$, $E_6$. Therefore, we cannot choose $X_4$ instead of $B_4$. Besides, $B_6$ cannot be replaced by $X_6$ because $\Delta^2E_4 X_6$ can be expressed as a polynomial in our generators $A_1$, $A_2$, $A_3$, $A_4$, $A_5$, $B_2$, $B_3$, $B_4$ and $E_4$, $E_6$.
\end{remark}

\section{ \texorpdfstring{$W(E_8)$}{W(E8)}-invariant Jacobi forms of small index}\label{Sec:5}

It is well-known that the space $J^{\w ,W(E_8)}_{*, E_8,1}=J^{W(E_8)}_{*, E_8,1}$ of $W(E_8)$-invariant weak (or holomorphic) Jacobi forms of index $1$ is a free module over $M_*$ generated by the theta function  $\vartheta_{E_8}$. In this big section, we give explicit descriptions of the structure of $J^{\w ,W(E_8)}_{*, E_8,t}$ and construct the generators when $t=2$, $3$, $4$. The cases of index $5$ and $6$ are also discussed. We develop two approaches to do this. The first one is based on the differential operators and the second relies on the pull-backs from $W(E_8)$-invariant Jacobi forms to the classical Jacobi forms for $A_1$.

\subsection{Notations and basic lemmas}\label{Subsec:5.1}
In this subsection we present a new way to characterize $q^0$-terms of $W(E_8)$-invariant Jacobi forms. This new way is convenient to calculate $q^0$-terms of Jacobi forms under the action of differential operators. Let us denote by $R_{2n}$ the set of all vectors $\ell \in E_8$ with $(\ell,\ell)=2n$. The Weyl group $W(E_8)$ acts on $R_{2n}$ in the usual way. The next lemma shows the orbits of $R_{2n}$ under the action of $W(E_8)$.

\begin{lemma}\label{lemorbit}
The orbits of $R_{2n}$ under the action of $W(E_8)$ are given by
\begin{align*}
&W(E_8)\backslash R_2=\{ w_8 \}& &W(E_8)\backslash R_4=\{ w_1 \}&\\
&W(E_8)\backslash R_6=\{ w_7 \}& &W(E_8)\backslash R_8=\{ 2w_8,w_2 \}&\\
&W(E_8)\backslash R_{10}=\{ w_1+w_8 \}& &W(E_8)\backslash R_{12}=\{ w_6 \}&\\
&W(E_8)\backslash R_{14}=\{w_3, w_7+w_8 \}& &W(E_8)\backslash R_{16}=\{2w_1, w_2+w_8 \}\\
&W(E_8)\backslash R_{18}=\{w_1+w_7, 3w_8 \}& &W(E_8)\backslash R_{20}=\{w_5, w_1+2w_8 \}\\
&W(E_8)\backslash R_{22}=\{w_6+w_8, w_1+w_2 \}& &W(E_8)\backslash R_{24}=\{2w_7, w_3+w_8 \}.&
\end{align*}
\end{lemma}

\begin{proof}
Applying the fact (a) in the proof of Lemma \ref{lem1}, we can prove the lemma by direct calculations.
\end{proof}

Corresponding to the above orbits, we define the following Weyl orbits.
\begin{align*}
&\sum_2=  \orb(w_8)& &\sum_4= * \orb(w_1)&
&\sum_6= * \orb(w_7)&\\
&\sum_{8'}= * \orb(w_2)&
&\sum_{8{''}}= * \orb(2w_8)& &\sum_{10}= * \orb(w_1+w_8)&\\
&\sum_{12}= * \orb(w_6)& &\sum_{{14}'}= * \orb(w_3)&
&\sum_{{14}{''}}= * \orb(w_7+w_8)&\\
&\sum_{{16}'}= * \orb(2w_1)&
&\sum_{{16}{''}}= * \orb(w_2+w_8)& &\sum_{{18}'}= * \orb(w_1+w_7)&\\
&\sum_{{18}{''}}= * \orb(3w_8)& &\sum_{{20}'}= * \orb(w_5)&
&\sum_{{20}{''}}= * \orb(w_1+2w_8)& \\
&\sum_{{22}'}= * \orb(w_1+w_2)&
&\sum_{{22}{''}}= * \orb(w_6+w_8)& &\sum_{{24}'}= * \orb(2w_7)&  \\
&\sum_{{24}{''}}= * \orb(w_3+w_8)& &\sum_{{26}'}= * \orb(2w_1+w_8)&  
&\sum_{{26}{''}}= * \orb(w_2+w_7)&\\ &\sum_{{28}'}= * \orb(w_1+w_6)& 
&\sum_{{30}'}= * \orb(w_4)& &\sum_{{32}'}= * \orb(w_1+w_3)&  \\
&\sum_{{32}{''}}= * \orb(2w_2)& &\sum_{{36}'}= * \orb(3w_1)&  
\end{align*}

The normalizations of these Weyl orbits are chosen such that they reduce to $240$ if one takes $\mathfrak{z}=0$. By Lemma \ref{lem2.3} and Equation \eqref{**}, it is easy to prove the next three lemmas.

\begin{lemma}\label{lem2.4}
We have the following estimations
\begin{align*}
&\max\left\{\min\{(v,v): v\in l+ 2E_8 \}: l\in E_8 \right\}=4,\\
&\max\left\{\min\{(v,v): v\in l+ 3E_8 \}: l\in E_8 \right\}=8,\\
&\max\left\{\min\{(v,v): v\in l+ 4E_8 \}: l\in E_8 \right\}=16,\\
&\max\left\{\min\{(v,v): v\in l+ 5E_8 \}: l\in E_8 \right\}=22,\\
&\max\left\{\min\{(v,v): v\in l+ 6E_8 \}: l\in E_8 \right\}=36.
\end{align*}
\end{lemma}

\begin{lemma}
Let $\varphi_t$ be a $W(E_8)$-invariant weak Jacobi form of index $t$. Then its $q^0$-term can be written as
\begin{align*}
[\varphi_2]_{q^0}=&240c_0 +c_1 \sum_{2} + c_2 \sum_{4},\\
[\varphi_3]_{q^0}=&240c_0 +c_1 \sum_{2} + c_2 \sum_{4}+c_3\sum_{6}+c_4\sum_{8'},\\
[\varphi_4]_{q^0}=&240c_0 +c_1 \sum_{2} + c_2 \sum_{4}+c_3\sum_{6}+c_4'\sum_{8'}+c_4{''}\sum_{8{''}}+c_5\sum_{10}+c_6\sum_{12}+c_7\sum_{{14}'}+c_8\sum_{{16}'},
\end{align*}
\begin{align*}
[\varphi_5]_{q^0}=&240c_0 +c_1 \sum_{2} + c_2 \sum_{4}+c_3\sum_{6}+c_4'\sum_{8'}+c_4{''}\sum_{8{''}}+c_5\sum_{10}+c_6\sum_{12}
+c_7'\sum_{{14}'}+c_7{''}\sum_{{14}{''}}\\
&+c_8'\sum_{{16}'}+c_8{''}\sum_{{16}{''}}+ c_9\sum_{{18}'}+c_{10}\sum_{{20}'}+c_{11}\sum_{{22}'},\\
[\varphi_6]_{q^0}=&240c_0 +c_1 \sum_{2} + c_2 \sum_{4}+c_3\sum_{6}+c_4'\sum_{8'}+c_4{''}\sum_{8{''}}+c_5\sum_{10}+c_6\sum_{12}+c_7'\sum_{{14}'}+c_7{''}\sum_{{14}{''}}\\&+c_8'\sum_{{16}'}+c_8{''}\sum_{{16}{''}}+ c_9'\sum_{{18}'}+c_9{''}\sum_{{18}{''}}+c_{10}'\sum_{{20}'}+c_{10}{''}\sum_{{20}{''}}+c_{11}'\sum_{{22}'}+c_{11}{''}\sum_{{22}{''}}+c_{12}'\sum_{{24}'}\\
&+c_{12}{''}\sum_{{24}{''}}+c_{13}'\sum_{{26}'}+c_{13}{''}\sum_{{26}{''}}+c_{14}\sum_{{28}'}+c_{15}\sum_{{30}'}+c_{16}'\sum_{{32}'}+c_{16}{''}\sum_{{32}{''}}+c_{18}\sum_{{36}'},
\end{align*}
where $c_i \in \CC$ are constants.
\end{lemma}

\begin{lemma}
\label{Lemtest}
Assume that $\varphi$ is a $W(E_8)$-invariant weak Jacobi form of index $t$.
\begin{enumerate}
\item Let $t=2$. Then $\varphi$ is a holomorphic Jacobi form if and only if its $q^0$-term is a constant. Moreover, $\varphi$ is a Jacobi cusp form if and only if its $q^0$-term is $0$ and its $q^1$-term is of the form $ c_0 +c_1 \sum_{2}$.
\item Let $t=3$. Then $\varphi$ is a holomorphic Jacobi form if and only if its $q^0$-term is a constant and its $q^1$-term is of the form
$$ 240c_0+c_1\sum_{2}+c_2 \sum_{4}+c_3\sum_{6}.$$
Moreover, a holomorphic Jacobi form $\varphi$ is a Jacobi cusp form if and only if $c_3=0$ and its $q^0$-term is $0$.
\item Let $t=4$. Then $\varphi$ is a holomorphic Jacobi form if and only if its $q^0$-term is a constant and its $q^1$-term is of the form
$$ 240c_0+c_1\sum_{2}+c_2 \sum_{4}+c_3\sum_{6}+c_4'\sum_{8'}+c_4{''}\sum_{8{''}}.$$
Moreover, a holomorphic Jacobi form $\varphi$ is a Jacobi cusp form if and only if $c_4'=c_4{''}=0$ and its $q^0$-term is $0$ and its $q^2$-term does not contain the term $\sum_{16'}$.
\item Let $t=5$. Then $\varphi$ is a holomorphic Jacobi form if and only if its $q^0$-term is a constant and its $q^1$-term is of the form
$$240c_0 +c_1 \sum_{2} + c_2 \sum_{4}+c_3\sum_{6}+c_4'\sum_{8'}+c_4{''}\sum_{8{''}}+c_5\sum_{10}$$
and its $q^2$-term does not contain the term $\sum_{22'}$.
Moreover, a holomorphic Jacobi form $\varphi$ is a Jacobi cusp form if and only if $c_5=0$ and its $q^0$-term is $0$ and its $q^2$-term does not contain the term $\sum_{20'}$.
\item Let $t=6$. Then $\varphi$ is a holomorphic Jacobi form if and only if its $q^0$-term is a constant and its $q^1$-term is of the form
$$240c_0 +c_1 \sum_{2} + c_2 \sum_{4}+c_3\sum_{6}+c_4'\sum_{8'}+c_4{''}\sum_{8{''}}+c_5\sum_{10}+c_6\sum_{12}$$
and its $q^2$-term does not contain the terms $\sum_{26'}$, $\sum_{26{''}}$, $\sum_{28'}$, $\sum_{30'}$, $\sum_{32'}$, $\sum_{32{''}}$, $\sum_{36'}$.
Moreover, a holomorphic Jacobi form $\varphi$ is a Jacobi cusp form if and only if $c_6=0$ and its $q^0$-term is $0$ and its $q^2$-term does not contain the terms $\sum_{24'}$, $\sum_{24{''}}$ and its $q^3$-term does not contain the term $\sum_{36'}$.
\end{enumerate}

\end{lemma}

We next explain how to determine holomorphic Jacobi forms of singular weight.
Let $\varphi_t$ be a $W(E_8)$-invariant holomorphic Jacobi form of weight $4$ and index $t$. In view of the singular weight, we have
$$ 
\varphi_{t}(\tau,\mathfrak{z})=\sum_{ n\in \NN}\sum_{\substack{\ell \in E_8\\ (\ell,\ell)=2nt}}f(n,\ell)e^{2\pi i (n\tau + (\ell,\mathfrak{z}))}.
$$
Therefore, the coefficients $f(n,\ell)$ depend only on the class of $\ell$ in $E_8/tE_8$.
Let $n\geq 1$ and  assume that
$$
\phi_t(\tau,\mathfrak{z})=q^n \sum_{\substack{\ell \in E_8\\ (\ell,\ell)=2nt}}f(n,\ell)e^{2\pi i (\ell,\mathfrak{z})}+O(q^{n+1})\in J^{W(E_8)}_{4,E_8,t}. 
$$
Since $\phi_t(\tau,0)=0$,
if there exists $\ell\in E_8$ such that $f(n,\ell)\neq 0$, then there exist $\ell_1,\ell_2 \in E_8$ satisfying $(\ell_1,\ell_1)=(\ell_2,\ell_2)=2nt$, $\orb(\ell_1)\neq \orb(\ell_2)$ and $(\ell_i,\ell_i)= \min\{(v,v): v\in \ell_i+tE_8  \}$, for $i=1,2$. From this, we deduce
\begin{align*}
&J^{W(E_8)}_{4,E_8,t}=\CC A_t, t=1,2,3,5,\\
&J^{W(E_8)}_{4,E_8,4}=\CC A_4 \oplus \CC X_4,\\
&1\leq \dim J^{W(E_8)}_{4,E_8,6} \leq 2.
\end{align*}
If $\dim J^{W(E_8)}_{4,E_8,6} =2$, then there exists a $W(E_8)$-invariant holomorphic Jacobi form of weight $4$ and index $6$ with Fourier expansion of the form
$$ 
F_{4,6}(\tau,\mathfrak{z})= q^2(\sum_{{24}'}-\sum_{{24}{''}})+O(q^3).
$$
Let $v_4$ be a vector of norm $4$ in $E_8$ and $z\in\CC$. By direct calculations, we see that   
$$
\frac{F_{4,6}(\tau,zv_4)}{\Delta^2(\tau)} =c_9\zeta^{\pm 9}+\sum_{1\leq j \leq 8} c(j)\zeta^{\pm j} +c(0) +O(q), \quad c_9\neq 0
$$
is a non-zero weak Jacobi form of weight $-20$ and index $12$ in the sense of Eichler and Zagier \cite{EZ}, where $\zeta=e^{2\pi i z}$ and $c(j)\in \CC$ are constants (see \S \ref{Subsec:5.6}). By \cite{EZ}, $J_{-20,12}^{\w}$ is generated by $E_4\phi_{-2,1}^{12}$ and $\phi_{-2,1}^{10}\phi_{0,1}^2$. Therefore there is a $\zeta^{\pm m}$ in the $q^0$-term of any weak Jacobi form of weight $-20$ and index $12$ satisfying $m\geq 11$, which leads to a contradiction. We have thus proved the following.
\begin{lemma}\label{Lem:singular}
\begin{align*}
&J^{W(E_8)}_{4,E_8,t}=\CC A_t, \quad t=1, 2, 3, 5,\\
&J^{W(E_8)}_{4,E_8,4}=\CC A_4 \oplus \CC X_4,\\
&J^{W(E_8)}_{4,E_8,6} =\CC X_6.
\end{align*}
\end{lemma}

\subsection{The case of index 2}\label{Subsec:5.2}
In this subsection we discuss the structure of the space of $W(E_8)$-invariant Jacobi forms of index $2$. Firstly, Theorem \ref{MTH} shows that $J^{W(E_8)}_{*,E_8,2}$ is a free $M_*$-module of rank $3$. It is obvious that $A_2$ and $B_2$ must be generators of weight $4$ and weight $6$ respectively. As $A_1^2$ and $E_4A_2$ are linearly independent, $A_1^2$ is a generator of weight $8$. Hence $J^{W(E_8)}_{*,E_8,2}$ is a free $M_*$-module generated by $A_2, B_2$ and $A_1^2$. This fact can also be proved by Corollary \ref{coro:index2}.

 If $\phi\in J^{\w ,W(E_8)}_{k,E_8,2}$, then $\Delta\phi \in J^{W(E_8)}_{k+12,E_8,2}$ by Lemma \ref{Lemtest}.  From $J^{W(E_8)}_{4,E_8,2}=\CC A_2$ and $J^{W(E_8)}_{6,E_8,2}=\CC B_2$, we obtain $k+12\geq 8$. Thus, $\dim J^{\w ,W(E_8)}_{k,E_8,2}=0$ for $k\leq -6$. 
We next construct many basic Jacobi forms of index 2.
\begin{align}
\varphi_{-4,2}&=\frac{\vartheta_{E_8}^2- \frac{1}{9} E_4 \left(\vartheta_{E_8} \lvert T_{-} (2)\right)}{\Delta} = 2\sum_{2} - \sum_{4}-240 + O(q)\in  J^{\w ,W(E_8)}_{-4, E_8,2} \\
\varphi_{-2,2}&= 3H_{-4}(\varphi_{-4,2})=\sum_{2} + \sum_{4}-480 + O(q)\in  J^{\w ,W(E_8)}_{-2, E_8,2} \\
\varphi_{0,2}&= \frac{1}{2}E_4 \varphi_{-4,2}- H_{-2}(\varphi_{-2,2})=\sum_{2} +120 + O(q)\in  J^{\w ,W(E_8)}_{0, E_8,2}
\end{align}

\begin{remark}There is another construction of $\varphi_{0,2}$
$$\varphi_{0,2}= * \sum_{\sigma \in W(E_8)} f(\tau, \sigma(\mathfrak{z})),$$
where $*$ is a constant and
$$ f(\tau, \mathfrak{z})= -\frac{[\vartheta(\tau,z_1+z_2)\vartheta(\tau,z_1-z_2) \cdots \vartheta(\tau,z_7+z_8)\vartheta(\tau,z_7-z_8)]\lvert T_{-}(2)}{\vartheta(\tau,z_1+z_2)\vartheta(\tau,z_1-z_2)\cdots \vartheta(\tau,z_7+z_8)\vartheta(\tau,z_7-z_8)}.$$
\end{remark}

It is easy to check the following constructions.
\begin{equation}
\begin{split}
A_2=&\frac{1}{9}\vartheta_{E_8}\lvert T_{-} (2)=\frac{8}{9}\Phi_{\Gamma_0(2),1,0}
=\frac{1}{1080} \left(3E_4 \varphi_{0,2}-E_4^2 \varphi_{-4,2}-E_6 \varphi_{-2,2}\right)\\
=&  1 + q\cdot  \sum_{4} +O(q^2).
\end{split}
\end{equation}
\begin{equation}
\begin{split}
B_2 =& \frac{16}{15}\Phi_{\Gamma_0(2),2E_2(2\tau)-E_2(\tau),0}
= \frac{1}{1080} \left(3E_6 \varphi_{0,2}-E_{4}E_{6} \varphi_{-4,2}-E_4^2 \varphi_{-2,2}\right)\\
=& 1 + q\left[-\frac{8}{5} \sum_{2}-\frac{3}{5} \sum_{4}+24 \right] \\
&+q^2\left[\sum_{8{''}}-\frac{24}{5}\sum_{8'}-\frac{224}{5}\sum_6-\frac{72}{5}\sum_4-\frac{32}{5}\sum_2+24 \right]+O(q^3)
\end{split}
\end{equation}

\begin{align}
 U_{12,2}&= \Delta \varphi_{0,2}= q\left(\sum_{2} +120\right) +O(q^2) \in J^{\cusp ,W(E_8)}_{12, E_8,2} \\
V_{14,2}&= \frac{1}{3}\Delta \left(E_6 \varphi_{-4,2}+E_4\varphi_{-2,2}\right) = q \left[\sum_{2} -240\right] +O(q^2) \in J^{\cusp ,W(E_8)}_{14, E_8,2}\\
 W_{16,2}&= \frac{1}{3}\Delta \left(E_4^2 \varphi_{-4,2}+E_6\varphi_{-2,2}\right) = q \left[\sum_{2} -240\right] +O(q^2) \in J^{\cusp ,W(E_8)}_{16, E_8,2}
\end{align}

It is easily seen that $\varphi_{-4,2}$ and $\varphi_{-2,2}$ must be generators of $J^{\w ,W(E_8)}_{*,E_8,2}$ over $M_*$. Since $[\varphi]_{q^0}(\tau,0)=0$ if $\varphi$ is a weak Jacobi form of negative weight, we claim that $\varphi_{0,2}$ is also a generator of $J^{\w ,W(E_8)}_{*,E_8,2}$ over $M_*$ due to $[\varphi_{0,2}]_{q^0}(\tau,0)=360$. We then arrive at the following structure theorem.

\begin{theorem}\label{Th index 2}
The spaces $J^{\w ,W(E_8)}_{*, E_8,2}$, $J^{W(E_8)}_{*, E_8,2}$ and $J^{\cusp ,W(E_8)}_{*, E_8,2}$ are all free $M_*$-modules generated by three Jacobi forms. More exactly, we have 
\begin{align*}
J^{\w ,W(E_8)}_{*, E_8,2}&=M_* \langle \varphi_{-4,2},\varphi_{-2,2},\varphi_{0,2} \rangle,\\
J^{W(E_8)}_{*, E_8,2}&=M_* \langle A_2,B_2,\vartheta_{E_8}^2 \rangle,\\
J^{\cusp ,W(E_8)}_{*, E_8,2}&=M_* \langle U_{12,2},V_{14,2},W_{16,2} \rangle.
\end{align*}
\end{theorem}
\begin{proof}
It remains to prove the third claim. The third claim can be covered by Corollary \ref{coro:index2}. But, we here use another way to prove it.  For arbitrary $f\in J^{W(E_8)}_{2k, E_8,2}$ with $k\geq 4$, there exist two complex numbers $c_1,c_2$ such that 
$$f-c_1E_{2k-4}A_2-c_2E_{2k-8}\Delta\varphi_{-4,2} \in J^{\cusp ,W(E_8)}_{2k, E_8,2},$$ 
we replace $E_{2k-8}\Delta\varphi_{-4,2}$ with $\Delta\varphi_{-2,2}$ when $k=5$. From this, we have
$$ \dim J^{\cusp ,W(E_8)}_{2k, E_8,2} = \dim J^{W(E_8)}_{2k, E_8,2} -2, \; k\geq 4.$$
We then assert that $\dim J^{\cusp ,W(E_8)}_{2k, E_8,2}=0$ for $k\leq 5$, and $\dim J^{\cusp ,W(E_8)}_{2k, E_8,2}=1$ for $2k=12,14$, and $\dim J^{\cusp ,W(E_8)}_{16, E_8,2}=2$. In view of the fact that $W_{16,2}$ is independent of $E_4U_{12,2}$, we complete the proof.
\end{proof}

As an application of our results, we prove that Wirthm\"{u}ller's theorem does not hold for $E_8$.

\begin{theorem}\label{polynomial}
The bigraded ring $J^{\w ,W(E_8)}_{*, E_8,*}$ over $M_*$ is not a polynomial algebra.
\end{theorem}

\begin{proof}
Suppose, contrary to our claim, that $J^{\w ,W(E_8)}_{*, E_8,*}$ is a polynomial algebra over $M_*$.  Then there exists a finite set $S$ such that $J^{\w ,W(E_8)}_{*, E_8,*}=M_*[S]$ and the elements of $S$ are algebraically independent over $M_*$. This contradicts the fact that $\vartheta_{E_8}$, $\varphi_{-4,2}$, $\varphi_{-2,2}$, $\varphi_{0,2} \in S$ and the following algebraic relation
$$ \vartheta_{E_8}^2 = \frac{1}{1080}E_4 \left(3E_4 \varphi_{0,2}-E_4^2 \varphi_{-4,2}-E_6 \varphi_{-2,2}\right)+\Delta \varphi_{-4,2}.$$
\end{proof}

\subsection{The case of index 3}\label{Subsec:5.3}
In this subsection we continue to discuss the structure of the module of $W(E_8)$-invariant Jacobi forms of index $3$. We first claim that the possible minimum weight of $W(E_8)$-invariant weak Jacobi forms of index $3$ is $-8$. If there exists a $W(E_8)$-invariant weak Jacobi form $\phi$ of weight $2k<-8$ and index $3$ whose $q^0$-term is not zero, then we can construct a weak Jacobi form of weight $-10$ and index $3$ whose $q^0$-term is not zero. In fact, this function can be constructed as $E_{-10-2k}\phi$ if $2k\leq -14$, or $H_{-12}(\phi)$ if $2k=-12$. We now assume that there exists a $W(E_8)$-invariant weak Jacobi form $\phi$ of weight $-10$ and index $3$ whose $q^0$-term is represented as
$$ [\phi]_{q^0}=240c_0 +c_1 \sum_{2} + c_2 \sum_{4}+c_3\sum_{6}+c_4\sum_{8'}.$$
Then $c_4\neq 0$, otherwise $\Delta\phi$ will be a $W(E_8)$-invariant holomorphic Jacobi form of weight $2$, which is impossible. By means of the differential operators, we construct $\phi_{-8}=H_{-10}(\phi)$, $\phi_{-6}=H_{-8}(\phi_{-8})$, $\phi_{-4}=H_{-6}(\phi_{-6})$ and $H_{-4}(\phi_{-4})$. They are respectively weak Jacobi forms of weight $-8$, $-6$, $-4$, $-2$ with $q^0$-term of the form (order: $240c_0, c_1\sum_2, c_2\sum_4, c_3\sum_6, c_4\sum_{8'}$)
\begin{align*}
&\text{weight} \quad -10: & (a_{1,j})_{j=1}^9=(1,1,1,1,1)\\
&\text{weight} \quad -10+2(i-1): &a_{i,j}=\left(\frac{18-2i}{12}-\frac{j-1}{3}\right) a_{i-1,j}
\end{align*}
where $2\leq i \leq 5$, $1\leq j \leq 5$. For these Jacobi forms, if we take $\mathfrak{z}=0$ then their $q^0$-terms will be zero.  We thus get a system of $5$ linear equations with $5$ unknowns
$$ Ax=0, \quad A=(a_{i,j})_{5\times 5}, \quad x=(c_0,c_1,c_2,c_3,c_4)^t.$$

By direct calculations, this system has only trivial solution, which contradicts our assumption. Hence the possible minimum weight is $-8$. Indeed, there exists the unique $W(E_8)$-invariant weak Jacobi form of weight $-8$ and index $3$ up to a constant.  Suppose that $\phi$ is a non-zero weak Jacobi form of weight $-8$ with $q^0$-term of the form
$$240c_0 +c_1 \sum_{2} + c_2 \sum_{4}+c_3\sum_{6} +c_4\sum_{8'}.$$ 
Similarly, we can construct weak Jacobi forms of weight $-6$, $-4$, $-2$ with $q^0$-term of the form
\begin{align*}
&\text{weight} \quad -8: & (b_{1,j})_{j=1}^9=(1,1,1,1,1)\\
&\text{weight} \quad -8+2(i-1): &b_{i,j}=\left(\frac{16-2i}{12}-\frac{j-1}{3}\right) b_{i-1,j}
\end{align*}
where $2\leq i \leq 4$, $1\leq j \leq 5$. Then we can build a system of 4 linear equations with 5 unknowns
\begin{equation}\label{S1}
Bx=0, \quad B=(b_{i,j})_{4\times 5}, \quad x=(c_0,c_1,c_2,c_3,c_4)^t.
\end{equation}

We find that $(c_0,c_1,c_2,c_3,c_4)=(1,-4,6,-4,1)$ is the unique nontrivial solution of the above system. Therefore, the weak Jacobi form of weight $-8$ and index $3$ is unique if it exists.  Next, we construct many weak Jacobi forms of index $3$.

\begin{align*}
B_{-2,3}=&-5\frac{\vartheta_{E_8}B_2-\frac{1}{28}E_6(\vartheta_{E_8}\lvert T_{-}(3))}{\Delta}
=3\sum_2+3\sum_4+5\sum_6-11\times 240+O(q) \in J^{\w ,W(E_8)}_{-2, E_8,3}\\
\varphi_{-4,3}=&\frac{\vartheta_{E_8}A_2-\frac{1}{28}E_4[\vartheta_{E_8}\lvert T_{-}(3)]}{\Delta}
=\sum_2+\sum_4-\sum_6-240+O(q)\in J^{\w ,W(E_8)}_{-4, E_8,3}\\
A_{0,3}=&\vartheta_{E_8}\varphi_{-4,2}=2\sum_2-\sum_4-240+O(q) \in J^{\w ,W(E_8)}_{0, E_8,3} \\
\varphi_{-2,3}=&3H_{-4}(\varphi_{-4,3})=\sum_2+\sum_6-480+O(q) \in J^{\w ,W(E_8)}_{-2, E_8,3}\\
\varphi_{0,3}=&\frac{3}{8}\left(A_{0,3}+E_4\varphi_{-4,3}-2H_{-2}(\varphi_{-2,3})  \right)=\sum_2+O(q) \in J^{\w ,W(E_8)}_{0, E_8,3}
\end{align*}

\begin{remark}
There is another construction of $\varphi_{0,3}$
$$\varphi_{0,3}= * \sum_{\sigma \in W(E_8)} g(\tau, \sigma(\mathfrak{z}))$$
where $*$ is a constant and the function $g$ is defined as 
$$ g(\tau, \mathfrak{z})= \prod_{i=1}^8 \frac{\vartheta(\tau,2z_i)}{\vartheta(\tau,z_i)} +\prod_{i=1}^8 \frac{\vartheta(\tau,2z_i)}{\vartheta_{00}(\tau,z_i)}+\prod_{i=1}^8 \frac{\vartheta(\tau,2z_i)}{\vartheta_{01}(\tau,z_i)}+\prod_{i=1}^8 \frac{\vartheta(\tau,2z_i)}{\vartheta_{10}(\tau,z_i)}. $$
\end{remark}

We next construct the $W(E_8)$-invariant weak Jacobi form of weight $-8$ and index $3$. 
Firstly, we can check 
$$ 
E_4^2\varphi_{-4,3}+6E_6\varphi_{-2,3}-2E_4A_{0,3}-E_6B_{-2,3}=O(q)\in J^{\w ,W(E_8)}_{4, E_8,3}.
$$
If $ E_4^2\varphi_{-4,3}+6E_6\varphi_{-2,3}-2E_4A_{0,3}-E_6B_{-2,3}=0$, then we have
\begin{align*}
f_{-6,3}&=\frac{E_4\varphi_{-4,3}-2A_{0,3}}{E_6}=-3\sum_2+3\sum_4-\sum_6+240+O(q)\in J^{\w ,W(E_8)}_{-6, E_8,3},\\
H_{-6}(f_{-6,3})&=-\frac{3}{2}\sum_2+\frac{1}{2}\sum_4+\frac{1}{6}\sum_6+200+O(q)\in J^{\w ,W(E_8)}_{-4, E_8,3}.
\end{align*}
It is easy to see that $f_{-6,3}$,  $H_{-6}(f_{-6,3})$, $\varphi_{-4,3}$ are free over $M_*$ because $E_4\phi_{-8,3}$, $H_{-6}(f_{-6,3})$ and $\varphi_{-4,3}$ are independent. However
$$
E_6f_{-6,3}-3E_4 H_{-6}(f_{-6,3})-\frac{3}{2}E_4\varphi_{-4,3}=O(q).
$$
Hence we can get a non-zero weak Jacobi form of index 3 and weight $-12$, which is impossible. It follows that  $ E_4^2\varphi_{-4,3}+6E_6\varphi_{-2,3}-2E_4A_{0,3}-E_6B_{-2,3}\neq 0$, and we can construct 
\begin{equation}
\begin{split}
\varphi_{-8,3}=&*\frac{E_4^2\varphi_{-4,3}+6E_6\varphi_{-2,3}-2E_4A_{0,3}-E_6B_{-2,3}}{\Delta}\\
=&\sum_{8'}-4\sum_6+6\sum_4-4\sum_2+240+O(q) \in J^{\w ,W(E_8)}_{-8, E_8,3}.
\end{split}
\end{equation}
\begin{equation}
\varphi_{-6,3}=-3H_{-8}(\varphi_{-8,3})=\sum_{8'}-6\sum_4+8\sum_2-720+O(q) \in J^{\w ,W(E_8)}_{-6, E_8,3}.
\end{equation}

In fact, we get the coefficients of the $q^0$-term of $\varphi_{-8,3}$ from the solution of the system of linear equations (\ref{S1}). We now arrive at our main theorem in this subsection.

\begin{theorem}\label{Th index 3}
The space $J^{\w ,W(E_8)}_{*, E_8,3}$ is a free $M_*$-module generated by five weak Jacobi forms. More precisely, we have 
$$ J^{\w ,W(E_8)}_{*, E_8,3}=M_* \langle \varphi_{-8,3},\varphi_{-6,3},\varphi_{-4,3},\varphi_{-2,3},\varphi_{0,3} \rangle.$$
\end{theorem}

\begin{proof}
We first claim that there is no weak Jacobi form of weight $-6$ and index $3$ independent of $\varphi_{-6,3}$.  Conversely, suppose that there exists a weak Jacobi form of weight $-6$ which is linearly independent of $\varphi_{-6,3}$, noted by $f$. Without loss of generality, we can assume  
$$
[f]_{q^0}=240c_0 +c_1 \sum_{2} + c_2 \sum_{4}+c_3\sum_{6} \neq 0.
$$ 
Once again, we can construct weak Jacobi forms of weight $-4$, $-2$ and $0$ by the differential operators, respectively. They have $q^0$-terms of the form (order: $240c_0, c_1\sum_2, c_2\sum_4, c_3\sum_6$)
\begin{align*}
&\text{weight} \quad -6: &(c_{1,j})_{j=1}^4=(1,1,1,1)\\
&\text{weight} \quad -6+2(i-1): &c_{i,j}=\frac{9-i-2j}{6}c_{i-1,j}
\end{align*}
where $2\leq i\leq 4$, $1\leq j \leq 4$. For each Jacobi form of negative weight, if we take $\mathfrak{z}=0$ then its $q^0$-term will be zero. Hence we have
$$ 
\sum_{j=1}^4 c_{i,j}c_{j-1}=0, \quad 1\leq i \leq 3.
$$
For the Jacobi form of weight zero, by Lemma \ref{lf}, we have 
$$ 
\sum_{j=1}^4 (12-6j)c_{4,j}c_{j-1} =0.
$$
We thus get a system of linear equations of $4\times 4$. By direct calculations, we obtain $c_j=0$ for $0\leq j \leq 3$, which contradicts our assumption.

Theorem \ref{MTH} shows that $J^{\w ,W(E_8)}_{*, E_8,3}$ is a free $M_*$-module generated by five weak Jacobi forms. It is obvious that $\varphi_{-8,3}$, $\varphi_{-6,3}$ and $\varphi_{-4,3}$ are generators. Since $\varphi_{-2,3}$ is independent of $E_6 \varphi_{-8,3}$ and $E_4 \varphi_{-6,3}$, the function $\varphi_{-2,3}$ must be a generator. Moreover, $\varphi_{0,3}$ is also a generator on account of $[\varphi_{0,3}]_{q^0}(\tau,0)\neq 0$. We then conclude the eager result.
\end{proof}

In the rest of this subsection, we investigate the spaces of holomorphic Jacobi forms and Jacobi cusp forms of index $3$. 
Let $k\geq 2$. It is easy to see that the five dimensional space $\mathfrak{A}$ generated by $E_{2k+8}\varphi_{-8,3}$, $E_{2k+6}\varphi_{-6,3}$, $E_{2k+4}\varphi_{-4,3}$, $E_{2k+2}\varphi_{-2,3}$, $E_{2k-4}\Delta\varphi_{-8,3}$ (if $k=3$, we replace $E_{2k-4}\Delta\varphi_{-8,3}$ with $\Delta\varphi_{-6,3}$) does not contain non-zero holomorphic Jacobi form of weight $2k$. Moreover, for any $\phi\in J^{\w ,W(E_8)}_{2k,E_8,3}$, there exists a Jacobi form $f\in \mathfrak{A}$ such that $\phi - f$ is a holomorphic Jacobi form. We then assert
\begin{equation*}
\dim J^{W(E_8)}_{2k, E_8,3}= \dim J^{\w ,W(E_8)}_{2k, E_8,3}-5, \quad k\geq 2.
\end{equation*}
It is clear that $\dim J^{W(E_8)}_{2k, E_8,3}= 1$, for $k=2,3$. Thus, we deduce
\begin{align*}
A_3=&\frac{1}{28}E_4(\vartheta_{E_8}\lvert T_{-}(3))=\frac{27}{28}\Phi_{\Gamma_0(3),1,0}=1+q \sum_6 +O(q^2),\\
B_3=& \frac{81}{160}\Phi_{\Gamma_0(3),3E_2(3\tau)-E_2(\tau),0}
=1+q\left[-\frac{7}{20}\sum_6 -\frac{27}{20}\sum_4-\frac{9}{20}\sum_2+12 \right]+O(q^2).
\end{align*}

We further construct
\begin{align}
A_2\vartheta_{E_8}&=1+q\left[\sum_2+\sum_4\right]+O(q^2) \in J^{W(E_8)}_{8, E_8,3}\\
B_2\vartheta_{E_8}&=1+q\left[-\frac{3}{5}\sum_2-\frac{3}{5}\sum_4+24\right]+O(q^2) \in J^{W(E_8)}_{10, E_8,3},\\
\vartheta_{E_8}^3&=1+3q\sum_2+O(q^2) \in J^{W(E_8)}_{12, E_8,3}.
\end{align}

It is easy to check that the following vector spaces have the corresponding basis.
\begin{align*}
&J^{W(E_8)}_{8, E_8,3}=\CC\{E_4A_3,\;A_2\vartheta_{E_8} \}\\
&J^{W(E_8)}_{10, E_8,3}=\CC\{E_6A_3,\;E_4B_3,\;B_2\vartheta_{E_8} \}\\
&J^{W(E_8)}_{12, E_8,3}=\CC\{E_4^2A_3,\;E_6B_3,\;E_4A_2\vartheta_{E_8},\;\vartheta_{E_8}^3 \}
\end{align*}

From the above discussions, we claim that $A_3,B_3,A_2\vartheta_{E_8},B_2\vartheta_{E_8},\vartheta_{E_8}^3$ are free over $M_*$. 
This proves the following theorem.

\begin{theorem}\label{Th index 3h}
The space $J^{W(E_8)}_{*, E_8,3}$ is a free $M_*$-module generated by five holomorphic Jacobi forms. More precisely, we have 
$$ J^{W(E_8)}_{*, E_8,3}=M_* \langle A_3,\;B_3,\;A_2\vartheta_{E_8},\;B_2\vartheta_{E_8},\;\vartheta_{E_8}^3 \rangle.$$
\end{theorem}

In the end, we determine the structure of Jacobi cusp forms of index 3. We first construct many basic Jacobi cusp forms.
\begin{align*}
U_{10,3}=&-\frac{35}{54}E_6A_3-\frac{50}{27}E_4B_3+\frac{5}{2}B_2\vartheta_{E_8}=
q\left[\sum_4-\frac{2}{3}\sum_2-80\right]+O(q^2)\in J^{\cusp ,W(E_8)}_{10, E_8,3}\\
U_{12,3}=&E_4A_2\vartheta_{E_8}-\vartheta_{E_8}^3
=q\left[\sum_4-2\sum_2+240\right]+O(q^2)\in J^{\cusp ,W(E_8)}_{12, E_8,3}\\
U_{14,3}=&\Delta(E_4\varphi_{-2,3}+E_6\varphi_{-4,3})
=q\left[\sum_4+2\sum_2-720\right]+O(q^2)\in J^{\cusp ,W(E_8)}_{14, E_8,3}\\
V_{12,3}=&\Delta\varphi_{0,3}=q\cdot\sum_2+O(q^2)\in J^{\cusp ,W(E_8)}_{12, E_8,3}\\
U_{16,3}=&\Delta^2\varphi_{-8,3}=O(q^2)\in J^{\cusp ,W(E_8)}_{16, E_8,3}
\end{align*}

For arbitrary $k\geq 4$, we can show that the two dimensional space $\mathfrak{B}$ generated by $E_{2k-4}A_3$ and $E_{2k-8}\Delta\varphi_{-4,3}$ (if $k=5$, we replace $E_{2k-8}\Delta\varphi_{-4,3}$ with $E_4B_3$) does not contain non-zero Jacobi cusp form of weight $2k$. Moreover, for any $\phi\in J^{W(E_8)}_{2k,E_8,3}$, there exists a Jacobi form $g\in \mathfrak{B}$ such that $\phi - g$ is a Jacobi cusp form. We thus deduce
\begin{equation}
\dim J^{\cusp ,W(E_8)}_{2k, E_8,3}= \dim J^{W(E_8)}_{2k, E_8,3}-2, \quad k\geq 4.
\end{equation}
In a similar argument, we prove the next theorem.

\begin{theorem}\label{thcusp3}
The space $J^{\cusp ,W(E_8)}_{*, E_8,3}$ is a free $M_*$-module generated by five Jacobi cusp forms. More exactly, we have 
$$ J^{\cusp ,W(E_8)}_{*, E_8,3}=M_* \langle U_{10,3},\; U_{12,3},\;V_{12,3},\;U_{14,3},\;U_{16,3}\;\rangle.$$
\end{theorem}

\subsection{The case of index 4}\label{Subsec:5.4}
In this subsection we study the space of $W(E_8)$-invariant Jacobi forms of index $4$. We first assert that the possible minimum weight of $W(E_8)$-invariant weak Jacobi forms of index $4$ is $-16$. If there exists a non-zero weak Jacobi form $\phi$ of index 4 and weight $k<-16$, then its $q^0$-term is not zero and is not of the form
$$c_4'\sum_{8'}+c_4{''}\sum_{8{''}}+240c_0,$$ 
otherwise we can construct a non-zero holomorphic Jacobi form of wight less than $4$ (i.e. $\Delta\phi$).
As in the case of index $3$, by the Eisenstein series and the differential operators, we can construct a weak Jacobi form of weight $-18$ with non-zero $q^0$-term and note it by $f$. For convenience, we write 
$$c_4'\sum_{8'}+c_4{''}\sum_{8{''}}=(c_4'+c_4{''})\sum_8=c_4\sum_8.$$ 

We can assume that $f$ has $q^0$-term of the form
$$ 240c_0 +c_1 \sum_{2} + c_2 \sum_{4}+c_3\sum_{6} +c_4\sum_{8} + c_5 \sum_{10} +c_6\sum_{12} +c_7\sum_{14'}+c_8\sum_{16'},$$
where $c_i$ are not all zero. We then construct weak Jacobi forms of weights $-16$, $-14$, $-12$, $-10$, $-8$, $-6$, $-4$, $-2$. They have $q^0$-terms of the form (order: $240c_0, c_1\sum_2, \cdots , c_8\sum_{16'}$)
\begin{align*}
&\text{weight} \quad -18: & (a_{1,j})_{j=1}^9=(1,1,1,1,1,1,1,1,1)\\
&\text{weight} \quad -18+2(i-1): &a_{i,j}=\frac{29-2i-3j}{12}a_{i-1,j}
\end{align*}
where $2\leq i \leq 9$, $1\leq j\leq 9$.  For these Jacobi forms, if we put $\mathfrak{z}=0$, then their $q^0$-terms will become zero. Hence we can get a system of linear equations: 
$$ Ax=0, \quad A=(a_{i,j})_{9\times 9}, \quad x=(c_0,c_1,\cdots,c_8)^t.$$
By direct calculations, we know that the determinant of the matrix $A$ is not zero, it follows that the $q^0$-term of $f$ is zero, which leads to  a contradiction. Hence the possible minimum weight is $-16$.  One weak Jacobi form of index $4$ and weight $-16$ can be constructed as
\begin{equation}
\begin{split}
\varphi_{-16,4} = &c \sum_{\sigma \in W(E_8)} h(\tau, \sigma(\mathfrak{z})) \\
=&\sum_{16'}-8\sum_{14'}+28\sum_{12}-56\sum_{10}+14\sum_{8{''}}+56\sum_{8'}-56\sum_{6}+28\sum_{4}-8\sum_{2}+240 + O(q),
\end{split}
\end{equation}
where $c$ is a constant and the function $h$ is defined as 
$$ h(\tau, \mathfrak{z})=\frac{1}{\Delta^2}\prod_{i=1}^{4}\vartheta(\tau,z_{2i-1}+z_{2i})^2\vartheta(\tau,z_{2i-1}-z_{2i})^2.$$

Next, we show that $\varphi_{-16,4}$ is the unique weak Jacobi form of index $4$ and weight $-16$ up to a constant. Similarly, suppose that there exists a weak Jacobi form $\phi_{-16,4}$ of weight $-16$ with $q^0$-term of the form
$$ 240c_0 +c_1 \sum_{2} + c_2 \sum_{4}+c_3\sum_{6} +c_4\sum_{8} + c_5 \sum_{10} +c_6\sum_{12} +c_7\sum_{14'}+c_8\sum_{16'}.$$
Then we can construct weak Jacobi forms of weight $-14$, $-12$, $-10$, $-8$, $-6$, $-4$, $-2$, respectively. They have $q^0$-terms of the form (order: $240c_0, c_1\sum_2, \cdots , c_8\sum_{16'}$)
\begin{align*}
&\text{weight} \quad -16: & (b_{1,j})_{j=1}^9=(1,1,1,1,1,1,1,1,1)\\
&\text{weight} \quad -16+2(i-1): &b_{i,j}=\frac{27-2i-3j}{12}b_{i-1,j}
\end{align*}
where $2\leq i \leq 8$, $1\leq j\leq 9$.  For these Jacobi forms, if we take $\mathfrak{z}=0$, then their $q^0$-terms will be zero.  We then get a system of linear equations: 
$$ Bx=0, \quad B=(b_{i,j})_{8\times 9}, \quad x=(c_0,c_1,\cdots,c_8)^t.$$
By direct calculations, it has the unique nontrivial solution
$$x=(1,-8,28,-56,70,-56,28,-8,1).$$ 
We see at once that $[\varphi_{-16,4}-\phi_{-16,4}]_{q^0}=*(\sum_{8'}-\sum_{8{''}})$. 
Thus $\Delta(\varphi_{-16,4}-\phi_{-16,4})$ is a holomorphic Jacobi form of weight $-4$, which yields $\varphi_{-16,4}=\phi_{-16,4}$.

Applying the differential operators to $\varphi_{-16,4}$, we construct the following basic weak Jacobi forms.

\begin{align*}
\varphi_{-14,4} =& -3H_{-16}(\varphi_{-16,4})\\
=& \sum_{16'}-2\sum_{14'}-14\sum_{12}+70\sum_{10}-28\sum_{8{''}}-112\sum_{8'}+
154\sum_{6}-98\sum_{4}+34\sum_{2}-1200 + O(q) 
\\
\varphi_{-12,4} =&-\frac{2}{7}H_{-14}(\varphi_{-14,4})-\frac{1}{7}E_4\varphi_{-16,4}\\
=& \sum_{14'}-4\sum_{12}+3\sum_{10}+2\sum_{8{''}}+8\sum_{8'}
-25\sum_{6}
+24\sum_{4}-11\sum_{2}+480 + O(q) \in J^{\w ,W(E_8)}_{-12, E_8,4}\\
\varphi_{-10,4} =&-\frac{4}{9}H_{-12}\left(\varphi_{-12,4}\right)-\frac{5}{162}
(E_4\varphi_{-14,4}-E_6\varphi_{-16,4})\\
=&\sum_{12}-4\sum_{10}+\sum_{8{''}}+4\sum_{8'}
-5\sum_{4}+4\sum_{2}-240 + O(q) \in J^{\w ,W(E_8)}_{-10, E_8,4}
\\
\varphi_{-8,4} =&-\frac{3}{5}H_{-10}(\varphi_{-10,4})-\frac{1}{15}E_4\varphi_{-12,4}+\frac{1}{90}E_6\varphi_{-14,4}-\frac{1}{90}E_4^2\varphi_{-16,4}\\
=&\sum_{10}-\frac{7}{10}\sum_{8{''}}-\frac{28}{10}\sum_{8'}+4\sum_{6}
-\sum_{4}-\sum_{2}+120 + O(q) \in J^{\w ,W(E_8)}_{-8, E_8,4}
\end{align*}

\begin{align*}
\varphi_{-6,4} =&-\frac{1}{2}E_4\varphi_{-10,4}+\frac{1}{6}
E_6\varphi_{-12,4}-\frac{1}{36}
(E_4^2\varphi_{-14,4}-E_{4}E_6\varphi_{-16,4})
-4H_{-8}(\varphi_{-8,4})\\
=&\sum_{8{''}}+4\sum_{8'}-14\sum_{6}
+12\sum_{4}-2\sum_{2}-240 + O(q) \in J^{\w ,W(E_8)}_{-6, E_8,4}
\\
\varphi_{-4,4} =&-\frac{10}{81}E_4\varphi_{-8,4}+\frac{5}{81}E_6\varphi_{-10,4}+\frac{5}{1458}
(E_{4}E_6\varphi_{-14,4}-E_{4}^3\varphi_{-16,4})\\
&-\frac{5}{243}
E_4^2\varphi_{-12,4}-\frac{2}{9}H_{-6}(\varphi_{-6,4})\\
=&\sum_{6}-2\sum_{4}+\sum_{2}+ O(q) \in J^{\w ,W(E_8)}_{-4, E_8,4}\\
\varphi_{-2,4} =&- \frac{5}{9}E_6\varphi_{-8,4}+\frac{5}{18}E_4^2\varphi_{-10,4}+\frac{5}{324}
(E_{4}^3\varphi_{-14,4}-E_{4}^2E_6\varphi_{-16,4})\\
&-\frac{5}{54}E_{4}E_6\varphi_{-12,4}+\frac{1}{6}E_4\varphi_{-6,4} +12H_{-4}(\varphi_{-4,4})\\
=&-7\sum_{4}+8\sum_{2}-240+ O(q) \in J^{\w ,W(E_8)}_{-2, E_8,4}\\
\varphi_{0,4} =&H_{-2}(\varphi_{-2,4})=2\sum_2 -120+ O(q) \in J^{\w ,W(E_8)}_{0, E_8,4}
\end{align*}

\begin{equation}
\begin{split}
\psi_{-8,4} =&\frac{\vartheta_{E_8}\lvert T_{-}(4)-73\vartheta_{E_8}(\tau,2\mathfrak{z})}{72\Delta}
=*\sum_{\sigma\in W(E_8)}\left[-\frac{1}{\Delta}\prod_{i=1}^8\vartheta(\tau,2z_i)\right](\tau,\sigma(\mathfrak{z}))\\
=&\sum_{8'}-\sum_{8{''}}+ O(q) \in J^{\w ,W(E_8)}_{-8, E_8,4}.
\end{split}
\end{equation}

We now arrive at our main theorem in this subsection.
\begin{theorem}\label{Th index 4}
The space $J^{\w ,W(E_8)}_{*, E_8,4}$ is a free $M_*$-module generated by ten weak Jacobi forms. More precisely, we have 
$$ J^{\w ,W(E_8)}_{*, E_8,4}=M_* \langle \varphi_{-2k,4},\; 0\leq k \leq 8; \;\psi_{-8,4} \rangle.$$
\end{theorem}
\begin{proof}
It is sufficient to show that there is no any other weak Jacobi form of weight less than $-4$ which is independent of $\varphi_{-2k,4},\; 0\leq k \leq 8$, and $\psi_{-8,4}$. We only prove that there is no weak Jacobi form of weight $-14$ independent of $\varphi_{-14,4}$ because other cases are similar. Suppose that there exists a weak Jacobi form of weight $-14$ which is linearly independent of $\varphi_{-14,4}$, noted by $f$. We assume  
$$
[f]_{q^0}=240c_0 +c_1 \sum_{2} + c_2 \sum_{4}+c_3\sum_{6} +c_4\sum_{8} + c_5 \sum_{10} +c_6\sum_{12} +c_7\sum_{14'}\neq 0.
$$ 
Once again, we can construct weak Jacobi forms of weight $-12$, $-10$, $-8$, $-6$, $-4$, $-2$ and $0$, respectively. They have $q^0$-terms of the following form (order: $240c_0, c_1\sum_2, \cdots , c_7\sum_{14'}$)
\begin{align*}
&\text{weight} \quad -14: & (c_{1,j})_{j=1}^8=(1,1,1,1,1,1,1,1,1)\\
&\text{weight} \quad -14+2(i-1): &c_{i,j}=\frac{25-2i-3j}{12}c_{i-1,j}
\end{align*}
where $2\leq i \leq 8$, $1\leq j \leq 8$.  For each Jacobi form of negative weight, if we put $\mathfrak{z}=0$ then its $q^0$-term will become zero. For the Jacobi form of weight zero, we can modify $c_{8,j}$ to $(14-6j)c_{8,j}$ by Lemma \ref{lf}. We then get a system of linear equations: 
$$ Cx=0, \quad C=(c_{i,j})_{8\times 8}, \quad x=(c_0,c_1,\cdots,c_7)^t.$$
By direct calculation, it has only the trivial solution. Therefore the $q^0$-term of $f$ is 
$$[f]_{q^0}= *\left(\sum_{8'}-\sum_{8{''}}\right).$$
Thus, $\Delta f$ is a holomorphic Jacobi form of weight $-2$ and then we have $*=0$, which contradicts our assumption.

At the end of the proof, we explain why there is no weak Jacobi form $\psi_{-6,4}$ of weight $-6$ with $[\psi_{-6,4}]_{q^0}= \sum_{8'}-\sum_{8{''}}$. If $\psi_{-6,4}$ exists, then $\psi_{-8,4}$, $\psi_{-6,4}$ and $\varphi_{-14,4}$ are free over $M_*$. This contradicts the fact that $(E_6\psi_{-8,4}-E_4\psi_{-6,4})/\Delta$ is in $J_{-14,E_8,4}^{\w ,W(E_8)}$.
\end{proof}

In the rest of this subsection, we study the spaces of holomorphic Jacobi forms and Jacobi cusp forms of index $4$. We first construct many basic Jacobi forms.

\begin{align}
A_4 =& \vartheta_{E_8}(\tau,2\mathfrak{z})=1+q\sum_{8{''}}+O(q^2)\\
B_4=&\frac{1}{33}B_2\lvert T_{-}(2)+\frac{2}{55}\Delta \varphi_{-6,4}= 1+q\left[\frac{1}{15}\sum_{8{''}}-\frac{28}{15}\sum_6-\frac{4}{15}\sum_2-8  \right]+O(q^2)
\end{align}

\begin{equation}
\begin{split}
C_{8,4}=&\frac{1}{54}\Delta(E_{4}^3\varphi_{-16,4}-E_{4}E_6\varphi_{-14,4}+6E_{4}^2\varphi_{-12,4}-18E_{6}\varphi_{-10,4}+36E_{4}\varphi_{-8,4})\\
=&q\left[\frac{1}{5}\sum_{8{''}}+\frac{4}{5}\sum_{8'}-4\sum_{6}
+6\sum_{4}-4\sum_{2}+240\right] + O(q^2) \in J^{W(E_8)}_{8, E_8,4}
\end{split}
\end{equation}

\begin{equation}
\begin{split}
U_{10,4}=&-\frac{5}{324}\Delta(E_{4}^2E_6\varphi_{-16,4}-E_{4}^3\varphi_{-14,4}+6E_{4}E_6\varphi_{-12,4}-18E_{4}^2\varphi_{-10,4}\\
&+36E_{6}\varphi_{-8,4}-\frac{54}{5}E_{4}\varphi_{-6,4}) -*\Delta^2\varphi_{-14,4}\\
=&q\left[\sum_{6}-3\sum_{4}+3\sum_{2}-240\right] + O(q^2) \in J^{\cusp ,W(E_8)}_{10, E_8,4}
\end{split}
\end{equation}

\begin{equation}
\begin{split}
U_{12,4}=&-\frac{5}{324}\Delta(E_{4}E_6^2\varphi_{-16,4}-E_{4}^2E_6\varphi_{-14,4}+6E_{6}^2\varphi_{-12,4}-18E_{4}E_6\varphi_{-10,4}\\
&+36E_{4}^2\varphi_{-8,4}-\frac{54}{5}E_{6}\varphi_{-6,4})
 -*\Delta^2E_4\varphi_{-16,4}\\
=&q\left[\sum_{6}-3\sum_{4}+3\sum_{2}-240\right] + O(q^2) \in J^{\cusp ,W(E_8)}_{12, E_8,4}
\end{split}
\end{equation}

Similar to the case of index $3$, we can show the following identities
\begin{align}
\dim J^{W(E_8)}_{2k, E_8,4}&= \dim J^{\w ,W(E_8)}_{2k, E_8,4}-13, \quad k\geq 2\\
\dim J^{\cusp ,W(E_8)}_{2k, E_8,4}&= \dim J^{W(E_8)}_{2k, E_8,4}-4, \quad k\geq 4,
\end{align}
and use them to prove the next theorem.
\begin{theorem}\label{Th index 4h}
\noindent
\begin{enumerate}
\item The space $J^{W(E_8)}_{*, E_8,4}$ is a free $M_*$-module generated by the following ten holomorphic Jacobi forms
\begin{align*}
&\operatorname{weight} \; 4:& &A_4&   &\Delta\psi_{-8,4}&\\
&\operatorname{weight}\; 6:& &B_4&   &\Delta\varphi_{-6,4}&\\
&\operatorname{weight} \; 8:& &C_{8,4}& &\Delta\varphi_{-4,4}&   &\Delta^2\varphi_{-16,4}&\\
&\operatorname{weight} \; 10:& &\Delta\varphi_{-2,4}& &\Delta^2\varphi_{-14,4}&\\
&\operatorname{weight} \; 12:& &\Delta^2\varphi_{-12,4}&
\end{align*}
\item The space $J^{\cusp ,W(E_8)}_{*, E_8,4}$ is a free $M_*$-module generated by the following ten Jacobi cusp forms
\begin{align*}
&\operatorname{weight} \; 8:&  &\Delta\varphi_{-4,4}-*\Delta^2\varphi_{-16,4}&\\
&\operatorname{weight}\; 10:& &\Delta\varphi_{-2,4}-*\Delta^2\varphi_{-14,4}&    &U_{10,4}& \\
&\operatorname{weight} \; 12:& &\Delta\varphi_{0,4}-*\Delta^2E_4\varphi_{-16,4}&    &U_{12,4}&  &\Delta^2\varphi_{-12,4}&\\
&\operatorname{weight} \; 14:& &\Delta^2(E_4\varphi_{-14,4}-E_6\varphi_{-16,4})& &\Delta^2\varphi_{-10,4}& \\
&\operatorname{weight} \; 16:& &\Delta^2\varphi_{-8,4}& &\Delta^2\psi_{-8,4}&
\end{align*}
\end{enumerate}
\end{theorem}

Remark that these constants $*$ are chosen to cancel the term $q^2\sum_{16'}$ in the above constructions of Jacobi cusp forms.

\subsection{Isomorphisms between spaces of Jacobi forms}\label{Subsec:5.5}
In this subsection we use the embeddings of lattices to build two isomorphisms between the spaces of certain Jacobi forms.  We denote by $J_{k,L,t}^{\w ,\Orth(L)}$ the space of weak Jacobi forms of weight $k$ and index $t$ for the lattice $L$ which are invariant under the integral orthogonal group $\Orth(L)$. The dual lattice of $L$ is denoted by $L^*$.

We first consider the case of index  $2$. Recall that the Nikulin's lattice is defined as (see \cite[Example 4.3]{GN18})
\begin{equation*}
N_8=\latt{8A_1, h=(a_1+\cdots + a_8)/2}\cong D_8^* (2),
\end{equation*}
where $(a_i,a_j)=2\delta_{ij}$, $(h,h)=4$. It is easy to check that $N_8$ is a sublattice of $E_8$. We then have
$$ N_8<E_8 \Rightarrow E_8 < N_8^* \Rightarrow E_8(2) < N_8^*(2)\cong D_8.$$
This embedding of lattices induces the following isomorphism.

\begin{proposition}
The natural map
\begin{align*}
J_{k,D_8,1}^{\w ,\Orth(D_8)} &\longrightarrow J_{k,E_8,2}^{\w ,W(E_8)}\\
\phi(\tau,Z) &\longmapsto \frac{1}{\abs{W(E_8)}}\sum_{\sigma\in W(E_8)} \widehat{\phi}(\tau,\sigma(\mathfrak{z}))
\end{align*}
is a $M_*$-modules isomorphism. Here, $Z=\sum_{i=1}^8z_ie_i$, the set $\{e_i: 1\leq i \leq 8\}$ is the standard basis of $\RR^8$, and $\widehat{\phi}(\tau,\mathfrak{z})$ is defined as
\begin{equation*}
\widehat{\phi}(\tau,\mathfrak{z})=\phi(\tau,z_1+z_2,z_1-z_2,z_3+z_4,z_3-z_4,z_5+z_6,z_5-z_6,z_7+z_8,z_7-z_8).
\end{equation*}
\end{proposition}
\begin{proof}
Under the discriminant groups, the $\Orth(D_8)$-orbit of $(\frac{1}{2}, ..., \frac{1}{2})$ corresponds to $\sum_4$. The $\Orth(D_8)$-orbit of $(1,0, ..., 0)$ corresponds to $\sum_2$. From this, we assert that the above map is injective by comparing $q^0$-terms of Jacobi forms.  The space $J_{*,D_8,*}^{\w ,\Orth(D_8)}$ is in fact the space of Weyl invariant weak Jacobi forms for the root system $C_8$. By Table \ref{tablewi},  there exist weak Jacobi forms $\phi_{-4,D_8,1}$, $\phi_{-2,D_8,1}$, $\phi_{0,D_8,1}$. Thus, by Theorem \ref{Th index 2}, we prove the surjectivity of the above map.
\end{proof}

Remark that 
$$
N_8<2D_4<E_8 \Rightarrow E_8(2)<2D_4^*(2)=2D_4<D_8
$$
and the induced map 
$$J_{k,2D_4,1}^{\w ,\Orth(2D_4)} \longrightarrow J_{k,E_8,2}^{\w ,W(E_8)}$$
is also a $M_*$-modules isomorphism.
 
We next consider the case of index $3$. We see from the extended Coxeter-Dynkin diagram of $E_8$
(see Figure \ref{figE8}) that $A_2\oplus E_6$ is a sublattice of $E_8$. Observing the extended Coxeter-Dynkin diagram of $E_6$ (see Figure \ref{figE6}), 
\begin{figure}[tb]
\noindent\[
\begin{tikzpicture}
\begin{scope}[start chain]
\foreach \dyni in {1,3,4,5,6} {
\dnode{\dyni}
}
\end{scope}
\begin{scope}[start chain=br going above]
\chainin (chain-3);
\dnodebr{2}
\dnodebr{E_6}
\end{scope}
\end{tikzpicture}
\]
\caption{Extended Coxeter-Dynkin diagram of $E_6$} 
\label{figE6}
\end{figure}
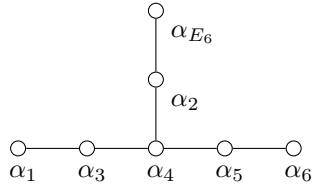
we find that $3A_2$ is a sublattice of $E_6$. We then have
$$ 4A_2<E_8 \Rightarrow E_8<4A_2^* \Rightarrow E_8(3)<4A_2^*(3)\cong 4A_2.$$

In a similar way, we can prove the following result.

\begin{proposition}
The natural map
\begin{align*}
J_{k,4A_2,1}^{\w ,\Orth(4A_2)} &\longrightarrow J_{k,E_8,3}^{\w ,W(E_8)}\\
\varphi(\tau,Z) &\longmapsto \frac{1}{\abs{W(E_8)}}\sum_{\sigma\in W(E_8)} \widetilde{\varphi}(\tau,\sigma(\mathfrak{z}))
\end{align*}
is a $M_*$-modules isomorphism. Here, we fix the standard model of $A_2$
\begin{align*}
A_2=\ZZ\beta_1 +\ZZ\beta_2,\; (\beta_i,\beta_i)=2,\; i=1,2,\; (\beta_1,\beta_2)=1,
\end{align*}
and $\widetilde{\varphi}(\tau,\mathfrak{z})$ is defined as
\begin{equation*}
\begin{split}
\widetilde{\varphi}(\tau,\mathfrak{z})=&\varphi(\tau,z_6-z_7,z_7+z_8,(z_1-z_2-z_3-z_4-z_5-z_6-z_7+z_8)/2,-z_1+z_2,\\
&-z_1-z_2,(z_1+z_2+z_3+z_4+z_5-z_6-z_7+z_8)/2,
-z_3+z_4,-z_4+z_5).
\end{split}
\end{equation*}
\end{proposition}

It is known that $J_{*,A_2,1}^{\w ,\Orth(A_2)}$ is generated by $a_{0,1}\in J_{0,A_2,1}^{\w ,\Orth(A_2)}$ and $a_{-2,1}\in J_{-2,A_2,1}^{\w ,\Orth(A_2)}$ over $M_*$ (see Table \ref{tablewi}). Using these two functions, we construct 
\begin{align*}
\varphi_{0,4A_2,1}&= a_{0,1}\otimes a_{0,1}\otimes a_{0,1}\otimes a_{0,1} \in J_{0,4A_2,1}^{\w ,\Orth(4A_2)},\\
\varphi_{-2,4A_2,1}&= \sum a_{-2,1}\otimes a_{0,1}\otimes a_{0,1}\otimes a_{0,1} \in J_{-2,4A_2,1}^{\w ,\Orth(4A_2)},\\
\varphi_{-4,4A_2,1}&= \sum a_{-2,1}\otimes a_{-2,1}\otimes a_{0,1}\otimes a_{0,1} \in J_{-4,4A_2,1}^{\w ,\Orth(4A_2)},\\
\varphi_{-6,4A_2,1}&= \sum a_{-2,1}\otimes a_{-2,1}\otimes a_{-2,1}\otimes a_{0,1} \in J_{-6,4A_2,1}^{\w ,\Orth(4A_2)},\\
\varphi_{-8,4A_2,1}&= a_{-2,1}\otimes a_{-2,1}\otimes a_{-2,1}\otimes a_{-2,1} \in J_{-8,4A_2,1}^{\w ,\Orth(4A_2)},
\end{align*}
here the sums take over all permutations of $4$ copies of $A_2$. Then we conclude from the above isomorphism that $J_{*,4A_2,1}^{\w ,\Orth(4A_2)}$ is generated by $\varphi_{-2j,4A_2,1}$, $0\leq j \leq 4$, over $M_*$. Moreover, the image of $\varphi_{-8,4A_2,1}$ gives a new construction of our generator of index 3 i.e. $\varphi_{-8,3}$.

We can also consider the case of index 5. It is easy to see $2A_4<E_8$, which yields $E_8(5)<2A_4^*(5)$. Unfortunately, the induced map
\begin{equation}\label{isom5}
J_{k,2A_4^*(5),1}^{\w ,\Orth(2A_4)} \longrightarrow J_{k,E_8,5}^{\w ,W(E_8)}
\end{equation} 
is not an isomorphism. In fact, for $\phi\in J_{k,2A_4^*(5),1}^{\w ,\Orth(2A_4)}$, the function $\Delta^2\phi$ is always a holomorphic Jacobi form. But, for $\psi\in J_{k,E_8,5}^{\w ,W(E_8)}$, the function $\Delta^2\psi$ is not a holomorphic Jacobi form in general. It is becauce that the lattice $A_4^*(5)$ satisfies the following $\operatorname{Norm}_2$ condition (see \cite[Lemma 2]{GW17} and \cite[Lemma 3.5]{GW18})
\begin{align*}
&\operatorname{Norm}_2:& &\forall \bar{c} \in L^\vee/L \quad \exists h_c \in \bar{c} : (h_c,h_c)\leq 2.
\end{align*}

\subsection{Pull-backs of Jacobi forms}\label{Subsec:5.6}
We have seen from \S \ref{Subsec:5.3} and \S \ref{Subsec:5.4} that our approach based on differential operators works well when the absolute value of minimal weight equals the maximal norm of Weyl orbits appearing in $q^0$-terms of Jacobi forms. But when the index is larger than $4$, the absolute value of minimal weight will be less than the maximal norm of Weyl orbits, which causes our previous approach to not work well because there are not enough linear equations in this case. In order to study $W(E_8)$-invariant Jacobi forms of higher index, we introduce a new approach relying on pull-backs of Jacobi forms.

For convenience, we first recall some results on classical Jacobi forms introduced by Eichler and Zagier in \cite{EZ}. Let $J_{2k,t}^{\w}$ be the space of weak Jacobi forms of weight $2k$ and index $t$. The ring of weak Jacobi forms of even weight and integral index is a polynomial algebra over $M_*$ generated by two basic weak Jacobi forms ($\zeta=e^{2\pi iz}$)
\begin{align*}
\phi_{-2,1}(\tau,z)&=\zeta+\zeta^{-1}-2+O(q) \in J_{-2,1}^{\w},\\
\phi_{0,1}(\tau,z)&=\zeta+\zeta^{-1}+10+O(q) \in J_{0,1}^{\w}.
\end{align*}

Let $\phi\in J_{k,E_8,t}^{\w ,W(E_8)}$ and $v_4$ be a vector of norm $4$ in $E_8$. Then 
$ 
\phi(\tau,zv_4)
$
is a weak Jacobi form of weight $k$ and index $2t$. In order to compute the Fourier coefficients of $\phi(\tau,zv_4)$, we consider the pull-backs of Weyl orbits. Let $\sum_v$ be a Weyl orbit associated to $v$ defined in \S \ref{Subsec:5.1}. Recall that
$$
\sum_v=\frac{240}{\abs{W(E_8)}} \sum_{\sigma\in W(E_8)} \exp (2\pi i (\sigma(v),\mathfrak{z})).
$$
Since $W(E_8)$ acts transitively on the set $R_4$ of vectors of norm $4$ in $E_8$ (see Lemma \ref{lemorbit}), we have
\begin{align*}
\sum_v(zv_4)=&\frac{240}{\abs{W(E_8)}} \sum_{\sigma\in W(E_8)} \exp (2\pi i (\sigma(v),v_4)z)
= \frac{240}{\abs{W(E_8)}} \sum_{\sigma\in W(E_8)} \exp (2\pi i (v,\sigma(v_4))z)\\
=& \frac{240}{\abs{R_4}}\sum_{l\in R_4} \exp (2\pi i (v,l)z).
\end{align*}
In view of this fact, we define 
\begin{align*}
\max ( \sum_v, v_4 )=\max ( v, R_4)= \max \{ (v,l): l\in R_4   \}.
\end{align*}
It is easy to check that 
\begin{align*}
& \max ( w_1, R_4 )= 4 &  & \max ( w_2, R_4 )= 5 & 
& \max ( w_3, R_4 )= 7 &  & \max ( w_4, R_4 )= 10 & \\
& \max ( w_5, R_4 )= 8 &  & \max ( w_6, R_4 )= 6 & 
& \max ( w_7, R_4 )= 4 &  & \max ( w_8, R_4 )= 2 & 
\end{align*}
and the maximal value can be obtained at $l=(0,0,0,0,0,0,0,2)$. Thus, we get

\begin{align*}
&\max ( \sum_2, v_4 )=2&    &\max ( \sum_4, v_4 )=4& 
&\max ( \sum_{6}, v_4 )=4&    &\max ( \sum_{8'}, v_4 )=5& \\
&\max ( \sum_{8{''}}, v_4 )=4&    &\max ( \sum_{10}, v_4 )=6& 
&\max ( \sum_{12}, v_4 )=6&    &\max ( \sum_{14'}, v_4 )=7& \\
&\max ( \sum_{14{''}}, v_4 )=6&    &\max ( \sum_{16'}, v_4 )=8& 
&\max ( \sum_{16{''}}, v_4 )=7&    &\max ( \sum_{18'}, v_4 )=8& \\
&\max ( \sum_{18{''}}, v_4 )=6&    &\max ( \sum_{20'}, v_4 )=8& 
&\max ( \sum_{20{''}}, v_4 )=8&    &\max ( \sum_{22'}, v_4 )=9& \\
&\max ( \sum_{22{''}}, v_4 )=8&    &\max ( \sum_{24'}, v_4 )=8& 
&\max ( \sum_{24{''}}, v_4 )=9&    &\max ( \sum_{26'}, v_4 )=10& \\
&\max ( \sum_{26{''}}, v_4 )=9&    &\max ( \sum_{28'}, v_4 )=10& 
&\max ( \sum_{30'}, v_4 )=10&    &\max ( \sum_{32'}, v_4 )=11& \\
&\max ( \sum_{32{''}}, v_4 )=10&    &\max ( \sum_{36'}, v_4 )=12.& 
\end{align*}

This new approach can be used to recover some cases of index $3$ and $4$.

\vspace{5mm}

\textbf{Index 3:} Assume that $\phi = \sum_{8'} + \cdots +O(q) \in  J_{-2k,E_8,3}^{\w ,W(E_8)}$, where $\cdots$ stands for the Weyl orbits of norm less than $8$. Then we have 
$$
\phi(\tau,zv_4)=\zeta^{\pm 5}+ \cdots +O(q) \in J_{-2k,6}^{\w},
$$
here $\cdots$ stands for the terms of type $\zeta^{\pm j}$ with $0\leq j\leq 4$. Note that in the above equation the term $\zeta^{\pm 5}$ may have positive coefficient different from $1$. But this does not affect our discussion. Thus, we always omit this type of coefficient hereafter.    We claim that $-2k\geq -8$. 
If $-2k<-8$ i.e. $k>4$, then $J_{-2k,6}^{\w}=\phi_{-2,1}^k\cdot J_{0,6-k}^{\w}$. Since $J_{-2k,6}^{\w}\neq \{0\}$, we have $k\leq 6$. But $J_{-10,6}^{\w}$ is generated by $\phi_{-2,1}^5\phi_{0,1}=\zeta^{\pm 6}+\cdots$ and $J_{-12,6}^{\w}$ is generated by $\phi_{-2,1}^6=\zeta^{\pm 6}+\cdots$. This contradicts the Fourier expansion of $\phi(\tau,zv_4)$.

Assume that $\phi \in  J_{-2k,E_8,3}^{\w ,W(E_8)}$ has no Fourier coefficient $\sum_{8'}$ in its $q^0$-term. By Lemma \ref{Lemtest}, we have $\Delta \phi \in J_{12-2k,E_8,3}^{W(E_8)}$. Thus we get $12-2k>4$ i.e. $-2k>-8$.

\textit{Conclusion}: The possible minimum  weight in this case is $\geq -8$ and the dimension of $J_{-8,E_8,3}^{\w ,W(E_8)}$ is at most one. 

\vspace{5mm}

\textbf{Index 4:} 
Assume that $\phi = \sum_{16'} + \cdots +O(q) \in  J_{-2k,E_8,4}^{\w ,W(E_8)}$ with $k>0$. Then we have 
$$
\phi(\tau,zv_4)=\zeta^{\pm 8}+ \cdots +O(q) \in J_{-2k,8}^{\w}.
$$
Since $J_{-2k,8}^{\w}=\phi_{-2,1}^k\cdot J_{0,8-k}^{\w}$, we have $8-k\geq 0$ i.e. $-2k\geq -16$.
 
Assume that $\phi \in  J_{-2k,E_8,4}^{\w ,W(E_8)}$ has no Fourier coefficient $\sum_{16'}$ in its $q^0$-term. By Lemma \ref{lem2.3} and Lemma \ref{lem2.4}, the function $\eta^{42} \phi$ is a  $W(E_8)$-invariant holomorphic Jacobi form of weight $21-2k$ and index $4$ with a character. In view of the singular weight, we get $21-2k\geq 4$ i.e. $-2k\geq -16$. 

\textit{Conclusion}:  The possible minimum  weight in this case is $\geq -16$.

Assume that $\phi = \sum_{14'} + \cdots +O(q) \in  J_{-2k,E_8,4}^{\w ,W(E_8)}$ with $k>0$. Then we have 
$$
\phi(\tau,zv_4)=\zeta^{\pm 7}+ \cdots +O(q) \in J_{-2k,8}^{\w}.
$$
Since $J_{-2k,8}^{\w}=\phi_{-2,1}^k\cdot J_{0,8-k}^{\w}$, the spaces $J_{-16,8}^{\w}$ and $J_{-14,8}^{\w}$ are all generated by one function with leading Fourier coefficient $\zeta^{\pm 8}$. Thus $k\leq 6$ i.e. $-2k\geq -12$.

Assume that $\phi \in  J_{-2k,E_8,4}^{\w ,W(E_8)}$ has no Fourier coefficients $\sum_{16'}$ and $\sum_{14'}$ in its $q^0$-term. By Lemma \ref{lem2.3} and Lemma \ref{lem2.4}, the function $\eta^{36} \phi$ is a  $W(E_8)$-invariant holomorphic Jacobi form of weight $18-2k$ and index $4$ with a character. In view of the singular weight, we get $18-2k\geq 4$ i.e. $-2k\geq -14$. 

\textit{Conclusion}: The dimension of $J_{-16,E_8,4}^{\w ,W(E_8)}$ is at most one.

\subsection{The case of index 5}\label{Subsec:5.7}
In this subsection we use the approach in \S \ref{Subsec:5.6} to determine the possible minimum weight of the generators of $J_{*,E_8,5}^{\w ,W(E_8)}$.

\textbf{(I)} Assume that $\phi = \sum_{22'} + \cdots +O(q) \in  J_{-2k,E_8,5}^{\w ,W(E_8)}$ with $k>0$. Then we have 
$$
\phi(\tau,zv_4)=\zeta^{\pm 9}+ \cdots +O(q) \in J_{-2k,10}^{\w}.
$$
Since $J_{-2k,10}^{\w}=\phi_{-2,1}^k\cdot J_{0,10-k}^{\w}$, we have $10-k\geq 0$. But when $k= 9$ or $10$, the spaces $J_{-20,10}^{\w}$ and $J_{-18,10}^{\w}$ are all generated by one function with leading Fourier coefficient $\zeta^{\pm 10}$, which contradicts the Fourier expansion of $\phi(\tau,zv_4)$. Therefore, we get $k\leq 8$ i.e. $-2k\geq -16$.

\textbf{(II)} Assume that $\phi \in  J_{-2k,E_8,5}^{\w ,W(E_8)}$ has no Fourier coefficient $\sum_{22'}$ in its $q^0$-term. By Lemma \ref{Lemtest}, the function $\Delta^2\phi \in J_{24-2k,E_8,5}^{W(E_8)}$. By Lemma \ref{Lem:singular}, we have $24-2k\geq 6$ i.e. $-2k\geq -18$.

\textbf{(III)} Assume that $\phi \in  J_{-2k,E_8,5}^{\w ,W(E_8)}$ has no Fourier coefficients $\sum_{22'}$ and $\sum_{20'}$ in its $q^0$-term. By Lemma \ref{lem2.3} and Lemma \ref{lem2.4}, the function $\eta^{44}\phi $ is a $W(E_8)$-invariant Jacobi cusp form of weight $22-2k$ and index $5$ with a character. From the singular weight, it follows that $22-2k>4$ i.e. $-2k\geq -16$.

\textbf{(IV)} Assume that $\phi \in  J_{-2k,E_8,5}^{\w ,W(E_8)}$ has no Fourier coefficients $\sum_{22'}$, $\sum_{20'}$ and $\sum_{18'}$ in its $q^0$-term. By Lemma \ref{lem2.3} and Lemma \ref{lem2.4}, the function $\eta^{40}\phi $ is a $W(E_8)$-invariant Jacobi cusp form of weight $20-2k$ and index $5$ with a character. It follows that $20-2k>4$ i.e. $-2k\geq -14$.

By the discussions above, we get the following result.

\begin{proposition}
\begin{align*}
&\dim J_{-2k,E_8,5}^{\w ,W(E_8)}= 0, \quad  \text{if} \quad -2k\leq -20, \\
&\dim J_{-18,E_8,5}^{\w ,W(E_8)}\leq 1,\\
&\dim J_{-16,E_8,5}^{\w ,W(E_8)}\leq 3.
\end{align*}
Moreover, if the $W(E_8)$-invariant weak Jacobi form of weight $-18$ and index $5$ exists, then its $q^0$-term has no Fourier coefficient $\sum_{22'}$ but must contain Fourier coefficient $\sum_{20'}$.
\end{proposition}

We do not know if the $W(E_8)$-invariant weak Jacobi form of weight $-18$ and index $5$ exists. But the $W(E_8)$-invariant weak Jacobi forms of weight $-16$ and index $5$ do indeed exist. We next show how to construct one such Jacobi form. We first construct a weak Jacobi form of weight $-8$ and index $1$ for the lattice $A_4^*(5)$ invariant under the orthogonal group $\Orth(A_4)$. It is known that 
$$
M_2\left(\Gamma_0(5),\left(\frac{\cdot}{5}\right)\right)=\CC\eta^5(\tau)/\eta(5\tau)+\CC\eta^5(5\tau)/\eta(\tau).
$$ 
Using the analogue of Proposition \ref{Prop:lift}, we can construct two independent holomorphic Jacobi forms $f_1$ and $f_2$ of weight $4$ from the space $M_2(\Gamma_0(5),(\frac{\cdot}{5}))$. Then the function $(*f_1-*f_2)/\Delta$ will be a weak Jacobi form of weight $-8$ for $A_4^*(5)$. 
The tensor product of two copies of this weak Jacobi form defines a weak Jacobi form of weight $-16$ and index 1 for $2A_4^*(5)$, whose image under the  map (\ref{isom5}) gives a $W(E_8)$-invariant weak Jacobi form of weight $-16$ and index 5 with the following $q^0$-term
\begin{equation}
\begin{split}
\varphi_{-16,5}=&240 -6 \sum_{2} +13 \sum_{4}-8\sum_{6} -14\sum_{8} + 28 \sum_{10} -14\sum_{12}\\
&-8\sum_{14}+13\sum_{16}-6\sum_{{18}'}+\sum_{{20}'},
\end{split}
\end{equation}
here and subquently, we use the following notations
$$
c_j{'}\sum_{2j'}+c_j{''}\sum_{2j{''}}=(c_j{'}+c_j{''})\sum_{2j}=c_j\sum_{2j}, \quad j=4,7,8,9,10,11,12,13.
$$
In general, we only know the coefficients $c_j$ and it is hard to calculate the exact values of $c_j'$ and $c_j{''}$. Thus, the above notations are convenient for us.

\subsection{The case of index 6}
In this subsection we discuss the possible minimum weight of the generators of $J_{*,E_8,6}^{\w ,W(E_8)}$.

\textbf{(I)} Assume that $\phi = \sum_{36'} + \cdots +O(q) \in  J_{-2k,E_8,6}^{\w ,W(E_8)}$ with $k>0$. Then we have 
$$
\phi(\tau,zv_4)=\zeta^{\pm 12}+ \cdots +O(q) \in J_{-2k,12}^{\w}.
$$
From $J_{-2k,12}^{\w}=\phi_{-2,1}^k\cdot J_{0,12-k}^{\w}$, we obtain $12-k\geq 0$ i.e. $-2k\geq -24$.

\textbf{(II)} Assume that $\phi = \sum_{32'}+*\sum_{32''} + \cdots +O(q) \in  J_{-2k,E_8,6}^{\w ,W(E_8)}$ with $k>0$. Similarly, we have 
$$
\phi(\tau,zv_4)=\zeta^{\pm 11}+ \cdots +O(q) \in J_{-2k,12}^{\w},
$$
and then $k\leq 12$. But when $k=11$ or $12$, the space $J_{-2k,12}^{\w}$ is generated by one function with leading $q^0$-term $\zeta^{\pm 12}$, which gives a contradiction. Thus, $k\leq 10$ i.e. $-2k\geq -20$.

\textbf{(III)} Assume that $\phi \in  J_{-2k,E_8,6}^{\w ,W(E_8)}$ has no Fourier coefficient $\sum_{36'}$ in its $q^0$-term. Then the function $\eta^{64}\phi$ is a $W(E_8)$-invariant holomorphic Jacobi form of weight $32-2k$ and index $6$ with a character. Hence, $32-2k\geq 4$. But when $2k=28$, the Jacobi form $\eta^{64}\phi$ has singular weight, which yields that $\eta^{64}\phi$ has a Fourier expansion of the form
$$
q^{8/3}\left(\sum_{32'}-\sum_{32{''}} \right)+O(q^{11/3}).
$$
This contradicts the above argument \textbf{(II)}. Thus, we have $-2k\geq -26$.

\textbf{(IV)} Assume that $\phi \in  J_{-2k,E_8,6}^{\w ,W(E_8)}$ has no Fourier coefficients $\sum_{36'}$, $\sum_{32'}$ and $\sum_{32{''}}$ in its $q^0$-term. Then the function $\eta^{60}\phi$ is a $W(E_8)$-invariant holomorphic Jacobi form of weight $30-2k$ and index $6$ with a character. Hence, $30-2k\geq 4$ i.e. $-2k\geq -26$. Similarly, when $2k=26$, the Jacobi form $\eta^{60}\phi$ has singular weight, which implies that $\eta^{60}\phi$ has a Fourier expansion of the form
$$
q^{5/2}\sum_{30'}+O(q^{7/2}).
$$
This is impossible because $\phi(\tau,0)=0$. Thus, we have $-2k\geq -24$.

\textbf{(V)} Assume that $\phi \in  J_{-2k,E_8,6}^{\w ,W(E_8)}$ has no Fourier coefficients $\sum_{36'}$, $\sum_{32'}$, $\sum_{32{''}}$ and $\sum_{30'}$ in its $q^0$-term. Then the function $\eta^{56}\phi$ is a $W(E_8)$-invariant holomorphic Jacobi form of weight $28-2k$ and index $6$ with a character. Hence, $-2k\geq -24$. Similarly, when $2k=24$, the Jacobi form $\eta^{56}\phi$ has singular weight, which forces that $\eta^{56}\phi$ has a Fourier expansion of the form
$$
q^{7/3}\sum_{28'}+O(q^{10/3}).
$$
This is impossible because $\phi(\tau,0)=0$. Thus, we have $-2k\geq -22$.

Combining the above arguments together, we have

\begin{proposition}
\begin{align*}
&\dim J_{-2k,E_8,6}^{\w ,W(E_8)}= 0, \quad  \text{if} \quad -2k\leq -28, \\
&\dim J_{-26,E_8,6}^{\w ,W(E_8)}\leq 1,\\
&\dim J_{-24,E_8,6}^{\w ,W(E_8)}\leq 3.
\end{align*}
Moreover, if the $W(E_8)$-invariant weak Jacobi form of weight $-26$ and index $6$ exists, then its $q^0$-term has no Fourier coefficients $\sum_{36'}$ and $\sum_{32'}$ but must contain Fourier coefficient $\sum_{32''}$.
\end{proposition}

It is easy to continue the above discussions and prove that
\begin{align*}
&\dim J_{-22,E_8,6}^{\w ,W(E_8)}\leq 4,&
&\dim J_{-20,E_8,6}^{\w ,W(E_8)}\leq 6.&
\end{align*}

We construct two independent $W(E_8)$-invariant weak Jacobi forms of weight $-24$ and index $6$. From the embeddings of lattices $4A_2<E_8$ and $2D_4<E_8$, we see 
\begin{align*}
&E_8(6)<4A_2(2),& &E_8(6)<2D_4(3).&
\end{align*}
By Table \ref{tablewi}, there exist an $\Orth(A_2)$-invariant weak Jacobi form $\phi_{-6,A_2,2}$ of weight $-6$ and index $2$ for $A_2$ and an $\Orth(D_4)$-invariant weak Jacobi form $\phi_{-12,D_4,3}$ of weight $-12$ and index $3$ for $D_4$.   The function $\phi_{-6,A_2,2}$ can be constructed as
$$
\frac{\vartheta^2(\tau,z_1)\vartheta^2(\tau,z_1-z_2)\vartheta^2(\tau,z_1)}{\eta^{18}(\tau)}
$$
and its $q^0$-term is rather simple.
As a tensor product of four copies of $\phi_{-6,A_2,2}$, we construct a $W(E_8)$-invariant weak Jacobi form of weight $-24$ and index $6$:
\begin{equation}
\begin{split}
\varphi_{-24,6}=&240 -8 \sum_{2} + 24 \sum_{4}-24\sum_{6}-36\sum_{8}+120\sum_{10}-88\sum_{12}
-88\sum_{14}+198\sum_{16}\\
&-88\sum_{18}-88\sum_{20}+120\sum_{22}-36\sum_{24}
-24\sum_{26}+24\sum_{28'}-8\sum_{30'}+\sum_{32{''}}.
\end{split}
\end{equation}

The function $\phi_{-12,D_4,3}$ can be constructed as a linear combination of basic Jacobi forms $\phi_{-4,D_4,1}^3$ and $\phi_{-4,D_4,1}\psi_{-4,D_4,1}^2$, where $\phi_{-4,D_4,1}$ is the generator of $J_{-4,D_4,1}^{\w ,W(D_4)}$ invariant under the odd sign change (i.e. $z_1\mapsto -z_1$) and $\psi_{-4,D_4,1}$ is the generator of $J_{-4,D_4,1}^{\w ,W(D_4)}$ anti-invariant under the odd sign change. We refer to \cite{Ber99} for their constructions. The $q^0$-term of $\phi_{-12,D_4,3}$ is quite complicated but it is easy to see that it contains only one $\Orth(D_4)$-orbit of vectors of norm $18$. Thus, the tensor product of $\phi_{-12,D_4,3}$ gives a $W(E_8)$-invariant weak Jacobi form of weight $-24$ and index $6$ with leading Fourier coefficient $\sum_{36'}$ in its $q^0$-term. We note this function by $\psi_{-24,6}$.

If the unique $W(E_8)$-invariant weak Jacobi form $\varphi_{-26,6}$ of weight $-26$ and index $6$ exists, then it is possible to prove that the above inequalities of dimensions will be equalities and the generators can be constructed by applying the differential operators to $\varphi_{-26,6}$, $\varphi_{-24,6}$ and $\psi_{-24,6}$. In addition, the function $\varphi_{-4,2}\varphi_{-16,4}$ should be also a generator of weight $-20$.

\subsection{Further remarks}
We close this section with the remarks below.

\begin{remark}
A holomorphic function is called a $E_8$ Jacobi form if it only satisfies the last three conditions in Definition \ref{def}. The space of $W(E_8)$-invariant Jacobi forms is much smaller than the space of $E_8$ Jacobi forms. By \cite{FM}, the dimension of the space of $E_8$ holomorphic Jacobi forms of weight 4 and index 2 is 51. It follows that the dimension of the space of $E_8$ weak Jacobi forms of weight -8 and index 2 is 50. One such function is 
$$\frac{1}{\Delta(\tau)}\prod_{i=1}^4 \vartheta(\tau,z_{2i-1}+z_{2i})\vartheta(\tau,z_{2i-1}-z_{2i}) .$$
But we have proved that there is no non-zero $W(E_8)$-invariant weak Jacobi form of weight -8 and index 2. We refer to \cite{EK} for the dimensional formulas of the spaces of $E_8$ Jacobi forms of large weight.
\end{remark}

\begin{remark}
The methods described in this paper would be useful to study the ring of Jacobi forms for other lattices.  The lattice $A_4^*(5)$ is of great importance in \cite{GW17,GW18}. Using our method, it is easy to determine the space of weak Jacobi forms of index $1$ for the lattice $A_4^*(5)$ which are invariant under the integral orthogonal group of $A_4^*(5)$. The space $J_{*,A_4^*(5),1}^{\w ,\Orth(A_4)}$ is a free module over $M_*$ generated by six weak Jacobi forms of weights $-8$, $-6$, $-4$, $-4$, $-2$ and $0$.
\end{remark}

\begin{remark}
One question has not been answered in this paper is if the bigraded ring of $W(E_8)$-invariant weak Jacobi forms is finitely generated over $M_*$. This question is at present far from being solved. If this ring is finitely generated, then it will contain a lot of generators (more than $20$) and there are plenty of algebraic relations among generators. Moreover, there might be generators of index larger than $6$ because the generator of degree $30$ of the ring of $W(E_8)$-invariant polynomials does not appear in the leading term of Taylor expansion of any $W(E_8)$-invariant weak Jacobi form of index at most $6$ (see Remark \ref{remark:Wir}).  
\end{remark}

\section{Modular forms in ten variables}\label{Sec:6}
In this section we investigate modular forms for the orthogonal group $\Orth^+(2U\oplus E_8(-1))$. In \cite{HU}, Hashimoto and Ueda proved that the graded ring of modular
forms for $\Orth^+(2U\oplus E_8(-1))$ is a polynomial ring in modular forms of weights $4$, $10$, $12$, $16$, $18$, $22$, $24$, $28$, $30$, $36$, $42$.  The dimension of the space of modular forms of fixed weight can be computed by their results. In \cite{DKW}, the authors showed that one can choose the additive liftings of Jacobi-Eisenstein series as generators. We next give an upper bound of $\dim M_k(\Orth^+(2U\oplus E_8(-1)))$ based on our theory of $W(E_8)$-invariant Jacobi forms. Since \cite{HU} has never been published, the proof may contain some gaps and thus our result is not inessential.

Let $F$ be a modular form of weight $k$ with respect to $\Orth^+(2U\oplus E_8(-1))$ with the trivial character and $\mathfrak{z}=(z_1,\cdots, z_8)\in E_8\otimes\CC$. We consider its Fourier-Jacobi expansion
\begin{align*}
F(\tau,\mathfrak{z},\omega)&=\sum_{\substack{n,m\in \NN, \ell\in E_8\\2nm-(\ell,\ell)\geq 0}}a(n,\ell,m)\exp(2\pi i(n\tau+(\ell,\mathfrak{z})+m\omega))\\
&=\sum_{m=0}^{\infty}f_m(\tau,\mathfrak{z})\xi^m,
\end{align*}
where $\xi=\exp(2\pi i \omega)$.
Then $f_m(\tau,\mathfrak{z})$ is a $W(E_8)$-invariant holomorphic Jacobi form of weight $k$ and index $m$ and we have the following symmetry
$$
a(n,\ell,m)=a(m,\ell,n),\quad \forall (n,\ell,m)\in \NN \oplus E_8 \oplus \NN. 
$$
For $r \in \NN$, we set
$$M_k(\Orth^+(2,10))(\xi^r)=\{F\in M_k(\Orth^+(2U\oplus E_8(-1))) : f_m=0,\; m<r\}$$
and 
$$ J_{k,E_8,m}^{W(E_8)}(q^r)=\{f \in J_{k,E_8,m}^{W(E_8)}: f(\tau,\mathfrak{z})=O(q^r) \}.$$
We then have the following exact sequence 
$$ 0\longrightarrow M_k(\Orth^+(2,10))(\xi^{r+1})\longrightarrow  M_k(\Orth^+(2,10))(\xi^r)\stackrel{P_r}\longrightarrow J_{k,E_8,r}^{W(E_8)}(q^r),$$
where $r\geq 0$ and $P_r$ maps $F$ to $f_r$.
From this, we obtain the following estimation
$$ \dim M_k(\Orth^+(2,10))(\xi^{r})-\dim M_k(\Orth^+(2,10))(\xi^{r+1}) \leq \dim J_{k,E_8,r}^{W(E_8)}(q^r). $$
It is clear that $J_{k,E_8,r}^{W(E_8)}(q^r)=\{0\}$ for sufficiently large $r$, and then
$$ M_k(\Orth^+(2,10))(\xi^{r})=M_k(\Orth^+(2,10))(\xi^{r+1})=\cdots=\{0\}$$
for sufficiently large $r$. We thus deduce
\begin{equation}\label{es1}
\dim M_k(\Orth^+(2U\oplus E_8)) \leq \sum_{r=0}^{\infty} \dim J_{k,E_8,r}^{W(E_8)}(q^r)\\
\leq \sum_{r=0}^{\infty} \dim J_{k-12r,E_8,r}^{\w ,W(E_8)}.
\end{equation}

We use the pull-back from $W(E_8)$-invariant Jacobi forms to $W(E_7)$-invariant Jacobi forms to improve inequality (\ref{es1}).

\begin{proposition}\label{prop5}
Let us fix the model $E_7=\{v\in E_8: (v,w_8)=0 \}$. We have the following homomorphism
\begin{align*}
\Phi:\; J_{2k,E_8,m}^{\w ,W(E_8)}&\longrightarrow J_{2k,E_7,m}^{\w ,W(E_7)},\\
f(\tau,\mathfrak{z})&\longmapsto f(\tau,z_1,\cdots,z_6,z_7,-z_7). 
\end{align*}
If $\Phi(f)=0$, then we have $f/G^2 \in J_{2k+240,E_8,m-60}^{\w ,W(E_8)}$, where 
$$G(\tau,\mathfrak{z})=\prod_{u\in R_2^+(E_8)}\frac{\vartheta(\tau,(u,\mathfrak{z}))}{\eta^3(\tau)}$$
is a weak $E_8$ Jacobi form of weight $-120$ and index $30$ which is anti-invariant under $W(E_8)$.
The symbol $R_2^+(E_8)$ denotes the set of all positive roots of $E_8$.

From the generators of $W(E_7)$-invariant Jacobi forms in Table \ref{tablewi}, we conclude
\begin{equation}
J_{2k,E_8,m}^{\w ,W(E_8)}=\{0\}\quad \textit{if}\quad 2k< -5m.
\end{equation}
\end{proposition}

As a consequence of the above proposition and (\ref{es1}), we deduce 
\begin{equation}\label{es2}
\dim M_{k}(\Orth^+(2U\oplus E_8(-1))) \leq \sum_{0\leq r \leq \frac{k}{7}}\dim J_{k-12r,E_8,r}^{\w ,W(E_8)}.
\end{equation}
We have $\dim M_{6}(\Orth^+(2U\oplus E_8(-1)))=0$ because $J_{6,E_8,1}^{W(E_8)}=\{0\}$.

By inequality \eqref{es2}, we get upper bounds of $\dim M_{2k}(\Orth^+(2U\oplus E_8(-1)))$. By \cite[Corollary 1.3]{HU}, we calculate the exact values of dimension. We list these datum in Table \ref{tabledim}.

\begin{table}[ht]
\caption{Dimension of orthogonal modular forms in ten variables}\label{tabledim}
\renewcommand\arraystretch{1.5}
\noindent\[
\begin{array}{|c|c|c|c|c|c|c|c|c|c|c|c|c|c|c|c|c|c|c|c|c|c|c|c|c|}
\hline 
\text{weight} & 4 & 6 & 8 & 10 & 12 & 14 & 16 & 18 & 20 & 22 & 24 & 26 &28 &30& 32& 34& 36& 38& 40& 42\\ 
\hline 
\text{bound}  & 1 & 0 & 1 & 1 & 2 & 1 & 3 & 2 & 4 & 4 & 6 & 5 & 9 & 8 & 12 & \textbf{13} & 17 & \textbf{17} & 24  & x  \\ 
\hline 
\text{dim.} & 1 & 0 & 1 & 1 & 2 & 1 & 3 & 2 & 4 & 4 & 6 & 5 & 9 & 8 & 12 & 12 & 17 & 16 & 24 & 23 \\
\hline 
\end{array}
\]
\end{table}
$$
x=23+\dim J_{-18,E_8,5}^{\w ,W(E_8)}.
$$
We see at once that the upper bound is equal to the exact dimension when the weight is less than $42$ and not equal to $34$ or $38$.

By means of the same technique, we can assert 
\begin{align*}
&\dim M_{4}(\Orth^+(2U\oplus E_8(-2)))=1,& &\dim M_{6}(\Orth^+(2U\oplus E_8(-2)))=1.&
\end{align*}
These two modular forms can be constructed as additive liftings of holomorphic Jacobi forms $A_2$ and $B_2$, respectively.

At the end of this section, we construct two reflective modular forms. Following \cite[Theorem 4.2]{G18}, we construct Borcherds products from weakly holomorphic Jacobi forms of weight $0$. It is easy to check that
\begin{align*}
&\phi_{0,2}=\frac{E_4^2A_2}{\Delta}+E_4\varphi_{-4,2}-2\varphi_{0,2}=q^{-1}+24+O(q)
\end{align*}
is a $W(E_8)$-invariant weakly holomorphic Jacobi form of weight $0$ and index $2$.  Its Borcherds product $\Borch(\phi_{0,2})\in M_{12}(\Orth^+(2U\oplus E_8(-2)),\det)$ is a reflective modular form of weight $12$ with  complete $2$-reflective divisor for $\Orth^+(2U\oplus E_8(-2))$. Another construction of this modular form can be found in \cite{GN18}.

The Borcherds product of $\varphi_{0,2}$ is a strongly reflective modular form of weight $60$ with respect to $\Orth^+(2U\oplus E_8(-2))$. Its divisor is determined by $2$-roots of $E_8$. It is anti-invariant with respect to the action of the Weyl group $W(E_8)$ and it is not an additive lifting. Moreover, its first Fourier--Jacobi coefficient is given by the theta block $\prod_{r\in R^+_2(E_8)}\vartheta(\tau, (r,\mathfrak{z}))$.

\bigskip

\noindent
\textbf{Acknowledgements}
This paper is a part of the author's doctoral thesis. The author is greatly indebted to his supervisor Valery Gritsenko for suggesting this problem and for many stimulating conversations. The author is also grateful to Kazuhiro Sakai for valuable discussions. This work was supported by the Labex CEMPI (ANR-11-LABX-0007-01) in the University of Lille.  The author thanks Max Planck Institute for Mathematics in Bonn for its hospitality and financial support. The author also thanks the referees for helpful comments and suggestions that make this paper better.

\bibliographystyle{amspalpha}

\end{document}